\documentclass[english]{amsart}

\usepackage[english]{babel}
\usepackage[T1]{fontenc}
\usepackage{lmodern}

\usepackage{MnSymbol}
\usepackage[centering]{geometry}

\theoremstyle{plain}
\newtheorem{theo}{Theorem}
\newtheorem{lemm}[theo]{Lemma}

\newtheorem{coro}[theo]{Corollary}

\theoremstyle{definition}
\newtheorem{defi}[theo]{Definition}

\theoremstyle{remark}
\newtheorem{rema}[theo]{Remark}

\usepackage{microtype}
\usepackage{enumerate}
\usepackage{graphicx}
\usepackage{color}
\usepackage{bm}
\usepackage{mathtools}
\usepackage[dvipsnames]{xcolor}
\definecolor{FlatRed}{RGB}{231,76,60}
\definecolor{FlatGreen}{RGB}{46,204,113}
\definecolor{FlatBlue}{RGB}{52,152,219}
\definecolor{FlatYellow}{RGB}{241,196,15}
\colorlet{FlatViolet}{FlatRed!50!FlatBlue}
\colorlet{FlatBrown}{FlatRed!50!FlatGreen}
\colorlet{FlatOrange}{FlatRed!50!FlatYellow}
\colorlet{FlatCyan}{FlatGreen!50!FlatBlue}

\usepackage{enumitem}
\usepackage{textcomp}
\usepackage[normalem]{ulem}

\usepackage{hyperref}
\hypersetup{
    colorlinks,
    linkcolor={FlatRed},
    citecolor={FlatGreen},
    urlcolor={FlatBlue}
}

\title[Tail profile of bulk Gaussian multiplicative chaos measures]{Bulk/boundary quotients of Gaussian multiplicative
chaos measures II: Tail profile of bulk Gaussian multiplicative chaos measures in the exact scale-invariant case}

\author{Yichao Huang}
\address{Beijing Institute of Technology, School of Mathematics and Statistics, Beijing, China}
\email{yichao.huang@bit.edu.cn}

\author{Youtao Liu}
\address{Beijing Institute of Technology, School of Mathematics and Statistics, Beijing, China}
\email{3120231467@bit.edu.cn}

\begin{document}

\begin{abstract}
This is the second part of a series of papers where we consider questions related to the tail profile of the bulk/boundary quotients of Gaussian multiplicative chaos measures appearing in boundary Liouville conformal field theory. In this part, we study of the right tail profile of a Gaussian multiplicative chaos measure with uniform singularity on the boundary, especially in the case of an exact scale-invariant kernel for the underlying log-correlated Gaussian field. This extends previous results by Rhodes-Vargas~\cite{Rhodes_2019} and Wong~\cite{wong2020universal}, where the case with flat background geometry (i.e. no boundary singularity) is studied using either the localization trick or some suitable versions of Tauberian theorems. Our generalization is non-trivial in the sense that we don't apply the localization trick to the bulk measure in question, but rather to an auxiliary Gaussian multiplicative chaos measure located at the boundary, on which the original bulk measure is not directly defined. We show that this modified localization scheme correctly captures the behavior of the right tail of the bulk Gaussian multiplicative chaos measure. The resulting tail profile coefficient is expressed in terms of a variant of bulk/boundary quotient of respective Gaussian multiplicative chaos measures, for which the well-definedness follows from preliminary joint moment bounds established previously in a companion paper~\cite{Huang:2025aa}.
\end{abstract}

\maketitle

\section{Introduction}
The goal of this paper is to continue our study of the following so-called bulk Gaussian multiplicative chaos measure, defined on some neighborhood of the origin of the upper half-plane model $\mathbb{H}=\{x+iy\in\mathbb{C}~;~y>0\}$ with parameter $\gamma\in(0,2)$:
\begin{equation}\label{eq:DefinitionBulkMeasure}
    \mu^{\mathrm{H}}(A)=\int_{A}\text{Im}(z)^{-\frac{\gamma^2}{2}}dM_{\mathrm{N}}(z)
\end{equation}
for any open measurable sets $A$ in the above neighborhood, and $dM_{\mathrm{N}}$ is the classical Gaussian multiplicative chaos measure (see Section~\ref{subse:gaussian_multiplicative_chaos}) associated to the covariance kernel
\begin{equation}\label{eq:GeneralCovKernel_O(1)}
    K_{\mathrm{N}}(z,w)=-\ln|z-w||z-\overline{w}|+O(1).
\end{equation}
This measure appears in boundary Liouville conformal field theory and its right tail is related to integrability results on Liouville conformal field theory, see~\cite{huang2018liouville,remy2022integrability,ang2024derivationstructureconstantsboundary}. Previously in~\cite{huang2018liouville,Huang:2023aa}, the following moment bound is obtained: for any small enough Carleson cube $Q_r=[-r,r]\times[0,2r]\subset\overline{\mathbb{H}}$,
\begin{equation}\label{eq:Huang23Result}
    \mathbb{E}\left[\mu^{\mathrm{H}}(Q_r)^{p}\right]<\infty
\end{equation}
if and only if $p<\frac{2}{\gamma^2}$.

From a probabilistic viewpoint, it is natural in view of~\cite{Barral_2013,Rhodes_2019,wong2020universal} to refine the above result into a more precise description of the right tail profile of $\mu^{\mathrm{H}}$:
\begin{equation*}
    \mathbb{P}[\mu^{\mathrm{H}}>t]\sim C_{(1,0)}t^{-\frac{2}{\gamma^2}}
\end{equation*}
as $t\to\infty$, where $C_{(1,0)}$ is some positive constant independent of $t$. See~\cite{Rhodes_2019} for an overview of such tail expansion in related problems. Our main result below yields furthermore an expression of $C_{(1,0)}$ in terms of some joint moment of bulk/boundary quotient of Gaussian multiplicative chaos.

\subsection{Main result of this paper}
In this paper, we establish the following precise right tail profile for the above bulk Gaussian multiplicative chaos measure. We focus on the so-called exact-scaling case where the correction term $O(1)$ in the covariance kernel~\eqref{eq:GeneralCovKernel_O(1)} is null.
\begin{theo}[Tail profile of bulk measure in the exact-scaling case]\label{th:main_result}
Let $Q_r$ the Carleson cube defined as before, and $\mu^{H}$ the bulk Gaussian multiplicative chaos defined in~\eqref{eq:DefinitionBulkMeasure} with the exact-scaling kernel
\begin{equation}\label{eq:ExactScalingNeumannKernel}
    K_{\mathrm{C}}(z,w)=-\ln|z-w||z-\overline{w}|.
\end{equation}
Then there exists some positive constant $C_{(1,0)}$ such that, as $t$ goes to $\infty$,
\begin{equation*}
    \mathbb{P}[\mu^{\mathrm{H}}(Q_r)>t]\sim C_{(1,0)}t^{-\frac{2}{\gamma^2}}.
\end{equation*}
Furthermore, we have the following expression for $C_{(1,0)}$ in terms of a variant of the $(\frac{2}{\gamma^2},1)$-joint moment of the bulk/boundary quotient of Gaussian multiplicative chaos measures:
\begin{equation*}
    C_{(1,0)}=2r\cdot(1-\frac{\gamma^2}{4})\mathbb{E}\left[\frac{I^{\mathrm{H}}(\infty)^{\frac{2}{\gamma^2}}}{I^{\partial}(\infty)}\right]
\end{equation*}
where $2r$ is the Euclidean boundary length of $Q_r$, i.e. $2r=|Q_r\cap\mathbb{R}|$,  and $I^{\mathrm{H}}(\infty)$, $I^{\partial}(\infty)$ are variants of respectively the bulk and boundary Gaussian multiplicative chaos measures defined in~\eqref{eq:Definitions_I} below.
\end{theo}

The notations appearing in the expression of $C_{(1,0)}$ will be reviewed in Section~\ref{subse:decomposition_of_the_gaussian_multiplicative_chaos_measure_around_the_singularity}, and its finiteness is a consequence of the main results in the companion paper~\cite[Theorem~1]{Huang:2025aa}. We will explain this connection in detail with Lemma~\ref{lemm:finiteness_of_the_tail_profile_constant}. We stress that the focus of this paper is not to establish the above theorem with general covariance structure in arbitrary dimension, nor to discuss the best regularity assumptions on the covariance kernel or on the boundary of the domain. Indeed, our method uses a radial decomposition of the Gaussian free field in dimension two, and we work with the exact-scaling case for the log-correlation kernel for simplicity. The main purpose here is to convey the message that the introduction of an auxiliary boundary Gaussian multiplicative chaos seems instrumental in the study of the tail profile of the bulk Gaussian multiplicative chaos measure, which is a new observation that we think worthwhile to illustrate in a simple case. In view of the conclusion of this theorem (and the applications it might have to the study of Liouville conformal field theory), it would seem useful to establish universal tail profiles of general bulk/boundary quotients in arbitrary dimensions, for a better understanding of the joint law of the bulk/boundary Gaussian multiplicative chaos measures. This aspect will be pursued in subsequent papers.

\subsection{On the novelty of our localization strategy}\label{subse:on_the_localization_strategy}
Our localization scheme is a careful modification of the original Rhodes-Vargas proposition. In \cite[Remark~2.5]{Rhodes_2019} it was claimed that the method therein (i.e. what we call the original localization trick in the current paper) can be applied to the Gaussian multiplicative chaos associated to Neumann boundary Gaussian free field, but the condition is that one has to be careful to avoid the boundary singularities. Indeed, in the case where the total bulk mass is concerned, our main theorem above shows that the boundary singularity changes the polynomial power in the right tail (from $t^{-\frac{4}{\gamma^2}}$ in~\cite{Rhodes_2019} to $t^{-\frac{2}{\gamma^2}}$ in Theorem~\ref{th:main_result}), and the main goal of this paper is to design a finer localization trick adapted to the higher-dimensional singularity. Similarly, in~\cite[Remark~1]{wong2020universal}, it was argued that point-wise singularities with small weights can be treated using the method therein, which requires a fine scaling argument around each singularity. With higher dimensional singularities, the same approach has to be modified appropriately. We now give a heuristic scaling argument of this phenomenon, which is similar to the one observed in~\cite{Huang:2023aa}.

The localization trick uses Girsanov transformation to pin down the contribution to the right tail of a Gaussian multiplicative chaos measure to a ``localized expression''. In the cases considered by~\cite{Rhodes_2019,wong2020universal}, the background geometry on $\Omega\subset\mathbb{R}^2$ is ``flat'' in the sense that the Gaussian multiplicative chaos measure should not be a priori concentrated on specific deterministic regions, therefore after the localization trick, the ``localized expression'' comes with a pointwise singularity where the location of the singularity is chosen uniformly on $\Omega$ (see Section~\ref{subse:localization_trick} below for a quick review).

In the case considered in this article, the background geometry on $\mathbb{D}\subset\mathbb{R}^2$ comes naturally with uniform one-dimensional singularities around the boundary $\mathbb{T}=\partial\mathbb{D}$ due to the blowup and the renormalization of the Neumann covariance kernel. If one applies the original localization trick above, one will see that the ``localized expression'' depends quite heavily on the distance between the boundary of the disk and the pointwise singularity inside $\mathbb{D}$ created by the localization trick. Handling this dependence is quite delicate and loses the elegance and simplicity of the localization trick. Therefore, our approach proposed in this paper to modify the localization trick so that the additional pointwise singularity appears directly on the boundary $\mathbb{T}$, which requires introducing an appropriate boundary Gaussian multiplicative chaos to perform the reweighting of the measure via the Girsanov transformation. The choice of the boundary Gaussian multiplicative chaos measure is naturally inspired by the expression in the boundary Liouville conformal field theory, where a bulk measure (the one we study in Theorem~\ref{th:main_result}) and a boundary measure (the auxiliary one that we introduce) are coupled. It turns out that modulo some additional technical difficulties handled in the companion paper~\cite{Huang:2025aa}, one can successfully implement the modified localization trick with the same elegance as in the previously mentioned papers. We choose to follow roughly the same structure of~\cite{Rhodes_2019} in the so-called exact-scaling case considered in the sequel.

\begin{rema}\label{rema:only_boundary_covariance_matters}
In fact, the correlation structure of the log-correlated Gaussian field with positive distance away from the boundary has no effect on the result of our theorem. Especially, in the half-plane model $\mathbb{H}$ or in the disk model $\mathbb{D}$, the tail profile constant is independent of the Gaussian multiplicative chaos measures on the respective regions $\{x+iy\in\mathbb{C}~;~y\geq \eta\}\subset\mathbb{H}$ or $\{|z|\leq 1-\eta\}\subset\mathbb{D}$ for any $\eta>0$. This is argued in the remark above: for the disk model, the Gaussian multiplicative chaos measure on $B(0,1-\eta)\subset\mathbb{D}$ is a classical one in dimension $2$, therefore has right tail of the form $Ct^{-\frac{\gamma^2}{4}}$ by~\cite{Rhodes_2019} or~\cite{wong2020universal}. Compared to the desired right tail in Theorem~\ref{th:main_result}, this contribution is negligeable when $t$ is large. See also Remark~\ref{rema:Final_Remark} for an explicit illustration.
\end{rema}

\subsection{Structure of this paper}
We only consider the so-called exact-scaling case in this note and leave the general case to a separate note where one should probably make use of a more general approach in the spirit of~\cite{wong2020universal}.

In Section~\ref{sec:mathematical_background}, we provide some necessary mathematical background.

In Section~\ref{sec:strategy_of_proof_and_some_heuristics}, we explain our strategy and provide some intuitions to the proof of our main result. More precisely, we explain how to modify the classical localization trick in the boundary case, and perform some calculation on some simplified models to explain why this modification should give the correct decay speed for the right tail.

In Section~\ref{sec:reduction_to_local_estimates}, we start to implement the localization trick at the boundary and show a preliminary result known as ``getting rid of the non-singularity''. This preliminary study effectively reduces the study of the right tail profile to a local study of the log-correlated field around the singularity created by the localization trick.

In Section~\ref{sec:expression_of_the_universal_constant}, we derive the tail profile constant and prove Theorem~\ref{th:main_result}, the main result of this note. Especially, we establish the finiteness of the espression of the tail profile constant in Theorem~\ref{th:main_result} using~\cite{Huang:2025aa}. Based on the localization trick of~\cite{Rhodes_2019} but modified to the boundary case, we will use Williams' decomposition to a drifted Brownian motion, and exploit the explicit law of the maximum of the latter to extract the expression for the tail profile constant.

Finally, in Section~\ref{sec:extension_to_non_exact_scaling_kernels}, we record the expression of the right tail profile constant in the case of some general kernels that are not exact scale-invariant.

\subsection*{Acknowledgements.} Y.H. is partially supported by National Key R\&D Program of China (No. 2022YFA1006300) and NSFC-12301164.

\section{Mathematical background}\label{sec:mathematical_background}

We will be primarily working with the half-plane $\mathbb{H}=\{x+iy\in\mathbb{C}~;~y>0\}$ with boundary $\partial\mathbb{H}=\mathbb{R}$. By a standard conformal transformation, this is equivalent (in the conformal sense) to the unit disk $\mathbb{D}\subset\mathbb{C}$ with boundary $\partial\mathbb{D}=\mathbb{T}$, but it is convenient to write down some scaling relations with the half-plane setting.

\subsection{Gaussian multiplicative chaos}\label{subse:gaussian_multiplicative_chaos}
The classical theory of Gaussian multiplicative chaos was introduced by Kahane~\cite{kahane1985chaos}. The main object is a random measure defined on some domain $\Omega\subset\mathbb{R}^{d}$ for integer $d\geq 1$, formally the renormalized exponential of the so-called log-correlated Gaussian field $X$ defined on $\Omega$. More precisely, starting with a log-correlated Gaussian field, that is a generalized centered Gaussian process indexed by $\Omega$ with definite positive covariance
\begin{equation*}
    \mathbb{E}[X(x)X(y)]=-\ln|x-y|+g(x,y)
\end{equation*}
with some bounded function $g$, the Gaussian multiplicative chaos is a random measure $M_{\gamma}$ with parameter $\gamma\in(0,\sqrt{2d})$, whose law is given by the following on bounded mesurable sets $A\subset\Omega$:
\begin{equation*}
    M_{\gamma}(A)=\lim\limits_{\epsilon\to 0}\int_{A}\exp\left(\gamma X_{\epsilon}(x)-\frac{\gamma^2}{2}\mathbb{E}[X_{\epsilon}(x)^2]\right)d\sigma(x),
\end{equation*}
with $\epsilon$ denoting some regularization of the generalized distribution $X$ and $\sigma(dx)$ some given measure, which for simplicity is assumed absolutely continuous with respect to the usual Lebesgue measure in this paper. The convergence of measures above is the usual weak-$\ast$ convergence. There are several more or less standard choices of the regularization (circle-average, smooth mollifier, white noise decomposition, Fourier cutoff etc.), but under mild regularity assumptions the limiting measure $M_{\gamma}$ is unique in law. We refer to the previous paper~\cite{Huang:2025aa} in this series for more notations, and the review~\cite{Rhodes_2014} or the more recent textbook~\cite{Berestycki2024aa} for an overview of the history and applications of the classical theory of Gaussian multiplicative chaos.

\subsection{Localization trick}\label{subse:localization_trick}
The so-called localization trick for Gaussian multiplicative chaos measures, to the best of our knowledge, first appeared in~\cite{Rhodes_2019}, where the Gaussian multiplicative chaos measure is used as a tilting measure to shift the background measure. The result of this tilting trick is a background measure with singularity concentrated on a point, and the intensity of this pointwise singularity is strong enough to reduce the study of the right tail of the Gaussian multiplicative chaos over the whole domain to a local study by zooming in at the location of this pointwise singularity. Combining with the local decomposition of the Gaussian field recalled in Sections~\ref{subse:decomposition_of_neumann_boundary_gaussian_free_field_at_the_boundary} and~\ref{subse:williams_decomposition}, the localization trick effectly pins down the right tail contribution in the case of classical Gaussian multiplicative chaos measures.

Let us illustrate the above paragraph in mathematical terms in the classical setting. For simplicity of the presentation, consider a classical Gaussian multiplicative chaos measure $M_{\gamma}$ with pure-log kernel (i.e. with correction term $g\equiv 0$) and flat background measure (i.e. with $d\sigma$ the usual Lebesgue measure) on a bounded domain $\Omega\subset\mathbb{R}^{d}$, and we focus on the right tail of $M_{\gamma}(\Omega)$, i.e. $\mathbb{P}[M_{\gamma}(\Omega)>t]$ for large $t$. The localization trick, in its original form of~\cite{Rhodes_2019}, consists of performing the following judicious Girsanov transformation on the Gaussian multiplicative chaos measure $M_{\gamma}(\Omega)$:
\begin{equation*}
    \mathbb{P}[M_\gamma(\Omega)>t]=\mathbb{E}\left[\frac{M_{\gamma}(\Omega)}{M_{\gamma}(\Omega)}\mathbf{1}_{\{M_{\gamma}(\Omega)>t\}}\right]=\int_{\Omega}\mathbb{E}\left[\frac{1}{M_{\gamma,v}(\Omega)}\mathbf{1}_{\{M_{\gamma,v}(\Omega)>t\}}\right]dv,
\end{equation*}
where we defined $M_{\gamma,v}$ as the localized Gaussian multiplicative chaos measure at $v\in\Omega$,
\begin{equation*}
    M_{\gamma,v}(\Omega)=\int_{\Omega}|z-v|^{-\gamma^2}dM_{\gamma}(z).
\end{equation*}

One observes that the above integral is roughly the same for each $v\in\Omega$, and the localized measure $M_{\gamma,v}(\Omega)$ is effectively concentrated around $v$ due to the singularity created by the Girsanov transformation. In~\cite[Section~3.5]{Rhodes_2019} (see also~\cite[Lemma~12]{wong2020universal}), it was shown that the right tail of $M_{\gamma}(\Omega)$ comes from analyzing the localized measure $M_{\gamma,v}$ on an arbitrarily small neighborhood $B(v,\rho)$ of the localized point $v$. Roughly speaking, this is explained by the extra singularity $|z-v|^{-\gamma^2}$ created by the Girsanov transform, in such a way that the mass and the right tail of $M_{\gamma,v}(\Omega)$ is strongly concentrated around $v\in\Omega$.

The current note is strongly inspired by this observation, although we need a different localization scheme, see Section~\ref{subse:on_the_localization_strategy} for a discussion on why the above localization trick in the classical case would not work in our setting, and see Section~\ref{sec:strategy_of_proof_and_some_heuristics} for a heuristic view on our adapted proof strategy. In Section~\ref{sec:reduction_to_local_estimates} we show that a modified localization scheme, taken into account the uniform boundary singularity of the bulk meaure $\mu^{\mathrm{H}}$, is effective at lowering the dimension of the singularity (thus ``localizing'' the problem) and is in this sense the correct one to use in our setting.

\subsection{Gaussian multiplicative chaos with boundary singularity}\label{subse:gaussian_multiplicative_chaos_with_boundary_singularity}
Since the introduction of the Gaussian multiplicative chaos with boundary singularity~\eqref{eq:DefinitionBulkMeasure} is already in~\cite{Huang:2025aa}, we give a slightly different presentation and refer to the above articles for equivalent descriptions.

Recall that we work with the exact-scaling covariance kernel~\eqref{eq:ExactScalingNeumannKernel},
\begin{equation*}
    K_{\mathrm{C}}(z,w)=-\ln|z-w||\overline{z}-w|
\end{equation*}
on a neighborhood of the origin in $\overline{\mathbb{H}}$. Since this kernel can be written as the sum of
\begin{equation*}
    -\ln\frac{|z-w|}{|\overline{z}-w|}-2\ln|\overline{z}-w|\eqqcolon K_{\mathrm{D}}+K_{\mathrm{R}},
\end{equation*}
the Neumann boundary condition Gaussian free field $X_{\mathrm{C}}$ associated to $K_{\mathrm{C}}$ can be written as the independent sum of a Dirichlet boundary condition Gaussian free field $X_{\mathrm{D}}$ (with covariance kernel $K_{\mathrm{D}}$) and the harmonic extension $\mathcal{P}(X_{\mathrm{R}})$ of a boundary Gaussian free field $X_{\mathrm{R}}$ defined on the boundary $\mathbb{R}$ (with covariance kernel $K_{\mathrm{R}}$ restricted to $\mathbb{R}$). Inside the bulk, the harmonic extension part $\mathcal{P}(X_{\mathrm{R}})$ is a regular (smooth) Gaussian field, so that there is no problem of defining its renormalized exponential with parameter $\gamma>0$, which we denote by $\sigma^{\mathcal{P}}_{\gamma}(dx)$ and is absolutely continuous with respect to the Lebesgue measure. The bulk Gaussian multiplicative chaos measure in~\eqref{eq:DefinitionBulkMeasure} is nothing but
\begin{equation*}
    \mu^{\mathrm{H}}(A)=\int_{A}\text{Im}(z)^{-\frac{\gamma^2}{2}}\sigma^{\mathcal{P}}_{\gamma}(x)dM_{\mathrm{D}}(z)
\end{equation*}
where now the Gaussian multiplicative chaos $dM_{\mathrm{D}}(z)$ does not have a blowing up covariance near the boundary (thus a more ``classical'' one as in the previous subsection).

We are left to explain the appearance of the extra factor $\text{Im}(z)^{-\frac{\gamma^2}{2}}$ in the display above and in~\eqref{eq:DefinitionBulkMeasure}. This is due to the discrepancy in the renormalization conventions between physics and mathematics communities: in physics the renormalization in the exponential of the Neumann boundary condition Gaussian free field is
\begin{equation*}
    \lim_{\epsilon\to 0}(2\epsilon)^{\frac{\gamma^2}{2}}\exp(\gamma X_\epsilon(z))d^{2}z
\end{equation*}
while in probability the convention is usually
\begin{equation*}
    \lim_{\epsilon\to 0}\exp(\gamma X_\epsilon(z)-\frac{\gamma^2}{2}\mathbb{E}[X_\epsilon(z)^2])d^{2}z.
\end{equation*}
The difference is exactly $\text{Im}(z)^{-\frac{\gamma^2}{2}}$ in the limit for the kernel~\eqref{eq:ExactScalingNeumannKernel}, see~\cite[Proposition~2.1]{huang2018liouville}.

\begin{rema}
The above discussion shows that we are actually looking at the right tail of the above-defined harmonic extension part of the Gaussian multiplicative chaos, since the Dirichlet part is quite regular near the boundary. We however keep using the Neumann covariance kernel of~\eqref{eq:GeneralCovKernel_O(1)} in the sequel (and in the rest of this series) since expressions of this kind are the usual ones appearing in the existing physics literature on boundary Liouville conformal field theory.
\end{rema}

\subsection{Brownian motion associated to the Neumann boundary condition Gaussian free field at the boundary}\label{subse:decomposition_of_neumann_boundary_gaussian_free_field_at_the_boundary}
We recall the fact that the circle average regularization of the Gaussian free field behaves nicely in two dimension, i.e. when $d=2$. In particular, in the case of the Neumann boundary condition Gaussian free field $X_{\mathrm{C}}$, we have the following semi-circle average regularization at the boundary. Denote by $Q(v,\rho)=B(v,\rho)\cap\overline{\mathbb{H}}$ the half-ball of radius $\rho>0$ and centered at $v\in\mathbb{R}$ in $\mathbb{H}$, and by $S(v,\rho)=\partial Q(v,\rho)\cap\mathbb{H}$ the semi-circle of radius $\rho>0$ and centered at $v\in\mathbb{R}$ in $\mathbb{H}$.

In the rest of this introduction, we take $v=0$: the general case follows by translation along the real line $\mathbb{R}$. In this case, we can use polar coordinate on $\mathbb{H}$ and write $z\in\overline{\mathbb{H}}$ as $z=e^{-t}e^{i\theta}$ with $t\in(-\infty,\infty]$ and $\theta\in[0,\pi]$.

\begin{defi}[Semi-circle average regularization at the boundary]
Consider the Gaussian free field $X_{\mathrm{C}}$ associated with the exact-scaling kernel~\eqref{eq:ExactScalingNeumannKernel}. The semi-circle average of the log-correlated field $X$ centered at $v=0$ is given by
\begin{equation*}
    B_t=\frac{1}{\sqrt{2}}\cdot\frac{1}{\pi}\int_{0}^{\pi}X(e^{-t}e^{i\theta})d\theta.
\end{equation*}
\end{defi}
The extra factor $\frac{1}{\sqrt{2}}$ is added so that $B_t$ evolves exactly as a standard Brownian motion, as we show in the following lemma.
\begin{lemm}
The above defined Gaussian process $(B_t)_{t\geq 0}$ is a standard Brownian motion.
\end{lemm}
\begin{proof}
A short proof follows from~\cite[Section~3.2]{Rhodes_2019} and Cardy's doubling trick. For concreteness, let us give an alternative proof by direct computation. With the definitions of $B_t$ above, one has, for $0\leq s\leq t$, 
\begin{equation*}
\begin{split}
    \mathbb{E}[B_sB_t]&=\frac{1}{2}\frac{1}{\pi^2}\int_{[0,\pi]^{2}}\mathbb{E}[X(e^{-s}e^{i\theta})X(e^{-t}e^{i\theta'})]d\theta d\theta'\\
    &=\frac{1}{\pi}\int_{0}^{\pi}\left(\frac{1}{\pi}\int_{0}^{2\pi}\ln|e^{-s}e^{i\theta}-e^{-t}e^{i\theta'}|d\theta'\right)d\theta\\
    &=2\frac{1}{\pi}\int_{0}^{\pi}\ln|e^{-s}e^{i\theta}-0|d\theta\\
    &=2\min(s,t).
\end{split}
\end{equation*}
In particular, $B_0=0$ and $(B_t)_{t\geq 0}$ is indeed a standard Brownian motion.
\end{proof}
By conformal transformation, one obtains a similar decomposition for $\mathbb{D}$ at its boundary $\mathbb{T}=\partial\mathbb{D}$ with respect to the usual hyperbolic geodesics, but it is more convenient to work with $\mathbb{H}$ since the notations are easier.

\subsection{Semi-circle polar decomposition}\label{subse:semi_circle_polar_decomposition}
We now study the so-called radial decomposition for the semi-circle average procedure. This has the same effect as the radial decomposition of the circle-average regularization for the Gaussian free field, which is effective at singling out Brownian motions from the Gaussian free field. We stress that this is something quite specific to the two-dimensional log-correlated field that we are considering.

We now take out the Brownian motion part $B_t$ defined in the previous subsection from the log-correlated field $X$, and consider the following equality in law (in the sense of Schwartz distributions):
\begin{equation*}
    X(e^{-t}e^{i\theta})=\sqrt{2}B_t+Y^{\mathrm{H}}(t,\theta)
\end{equation*}
where $B_t$ is the standard Brownian above, and $Y^{\mathrm{H}}(t,\theta)$ is a generalized centered Gaussian field \emph{independent of $B$} with covariance
\begin{equation*}
    \mathbb{E}[Y^{\mathrm{H}}(t,\theta)Y^{\mathrm{H}}(t',\theta')]=\ln\frac{(e^{-t}\vee e^{-t'})^2}{|e^{-t}e^{i\theta}-e^{-t'}e^{i\theta'}||e^{-t}e^{i\theta}-e^{-t'}e^{-i\theta'}|}.
\end{equation*}
We point out at once that the process $Y^{\mathrm{H}}(t,\theta)$ is, in law, translation invariant in $t$: indeed the above covariance only depends on $|t-t'|$. The independence between $B$ and $Y^{\mathrm{H}}$ can be checked via a direct computation of their covariance. The field $Y^{\mathrm{H}}$ is still log-correlated, thus we can consider its associated bulk Gaussian multiplicative chaos measure
\begin{equation*}
    \exp\left(\gamma Y^{\mathrm{H}}(t,\theta)-\frac{\gamma^2}{2}\mathbb{E}[Y^{\mathrm{H}}(t,\theta)^2]\right)dtd\theta
\end{equation*}
and the associated random generalized function
\begin{equation*}
    Z^{\mathrm{H}}_t=\int_{0}^{\pi}(\sin\theta)^{-\frac{\gamma^2}{2}}\exp\left(\gamma Y^{\mathrm{H}}(t,\theta)-\frac{\gamma^2}{2}\mathbb{E}[Y^{\mathrm{H}}(t,\theta)^2]\right)d\theta.
\end{equation*}
This decomposition is similar when the change to polar coordinates is centered at $v\neq 0$ by translation, but we do not record the formulas here.

Similarly, we can decompose the boundary measure $\mu^{\partial}$ using the same recipe as above, by considering the associated formal (boundary) Gaussian multiplicative chaos measure $Z^{\partial}$:
\begin{equation*}
    Z^{\partial}_t=\exp\left(\frac{\gamma}{2} Y^{\mathrm{H}}(t,0)-\frac{\gamma^2}{8}\mathbb{E}[Y^{\mathrm{H}}(t,0)^2]\right)+\exp\left(\frac{\gamma}{2} Y^{\mathrm{H}}(t,\pi)-\frac{\gamma^2}{8}\mathbb{E}[Y^{\mathrm{H}}(t,\pi)^2]\right).
\end{equation*}

\begin{rema}
It should be noted that the quantities $Z^{\mathrm{H}}_s, Z^{\partial}_s$ defined above should be understood in the distributional sense, i.e. they only make sense when appearing in expressions of type
\begin{equation*}
    \int f(s)Z^{\bullet}_sds
\end{equation*}
for some smooth function $f$ under the usual $\epsilon$-regularization procedure recalled in Section~\ref{subse:gaussian_multiplicative_chaos}. Furthermore, since $Y^{\mathrm{H}},Y^{\partial}$ are correlated, $Z^{\mathrm{H}}$ and $Z^{\partial}$ are not independent, although they are independent of the Brownian motion part $B$ defined in Section~\ref{subse:decomposition_of_neumann_boundary_gaussian_free_field_at_the_boundary}. We will automatically treat these quantities in this way so that they are non-degenerate when $\gamma\in(0,2)$.
\end{rema}

\subsection{Williams decomposition}\label{subse:williams_decomposition}
We recall the following decomposition theorem on the drifted Brownian motion with respect to its maximum. The statement is taken from~\cite[Lemma~3.4]{Rhodes_2019}.
\begin{lemm}[Williams decomposition theorem for drifted Brownian motion]
Let $(B_s)_{s\geq 0}$ be a standard Brownian motion, and $(B_s-\alpha s)_{s\geq 0}$ with $\alpha>0$ be a negatively drifted Brownian motion. Define $\widetilde{M}=\sup_{s\geq 0}(B_s-\alpha s)$. Then conditionally on $\widetilde{M}$, the law of the path $(B_s-\alpha s)_{s\geq 0}$ is the concatenation of two independent paths:
\begin{itemize}
    \item A Brownian motion with opposite positive drift $(\overline{B}_s+\alpha s)_{s\leq\tau_{\widetilde{M}}}$, run until its hitting time $\tau_{\widetilde{M}}$ of $\widetilde{M}$;
    \item A negatively drifted Brownian motion starting from $\widetilde{M}$ conditioned to stay negative $(\widetilde{M}+\widetilde{B}_{s-\tau_{\widetilde{M}}}+\alpha (s-\tau_{\widetilde{M}}))_{s\geq\tau_{\widetilde{M}}}$.
\end{itemize}
One also has the following time reversal property for all $C>0$ and $\tau_C$ the hitting time of $C$: this property can be expressed as an equality in law
\begin{equation*}
    (\overline{B}_{\tau_C-s}+\alpha(\tau_C-s)-C)_{s\leq \tau_C}=(\widehat{B}_s-\alpha s)_{s\leq L_{-C}},
\end{equation*}
where $(\widehat{B}_s-\alpha s)_{s\geq 0}$ is a Brownian motion with drift $-\alpha$ conditioned to stay negative, and $L_{-C}$ is the last time $(\widehat{B}_s-\alpha s)$ hits $-C$.
\end{lemm}

\subsection{Explicit calculations on the maximum of a drifted Brownian motion}\label{subse:explicit_calculations_on_the_maximum_of_a_drifted_brownian_motion}

Let $(B_s-\alpha s)_{s\geq 0}$ with $\alpha>0$ be a negatively drifted Brownian motion, and $\widetilde{M}=\sup_{s\geq 0}(B_s-\alpha s)$ its maximum defined in the previous section. We have the following explicit identity on the distribution of $\widetilde{M}$:
\begin{equation*}
    \forall t\geq 1,\quad \mathbb{P}[e^{\widetilde{M}}>t]=t^{-2\alpha}.
\end{equation*}
Notice that this identity depends on the drift $\alpha>0$, which will be chosen later.

\section{Strategy of proof and some heuristics}\label{sec:strategy_of_proof_and_some_heuristics}
We first provide some highlevel overviews on the strategy of proof, more precisely on the modified localization scheme that we refer to as the \emph{localization trick at the boundary}. We then supplement this strategy with some heuristical argument which yields the correct decaying speed $t^{-\frac{2}{\gamma^2}}$ for the right tail of the bulk measure as announced in Theorem~\ref{th:main_result}.

\subsection{Localization at the boundary}\label{subse:localization_at_the_boundary}
The key observation of this paper is that, although our goal is to establish the right tail estimate of the measure $\mu^{\mathrm{H}}(Q_r)$ over the cube $Q_r$, the contribution of the mass comes mostly from the part of the measure $\mu^{\mathrm{H}}$ near the boundary $I_r=[-r,r]\in\mathbb{R}$. This is alrealy observed in~\cite{huang2018liouville,Huang:2023aa}, with notably the following positive moment bound:
\begin{equation*}
    \mathbb{E}[\mu^{\mathrm{H}}(Q_r)^{p}]<\infty
\end{equation*}
if and only if $p<\frac{2}{\gamma^2}$. Recall that, contrary to the moment bound $\frac{4}{\gamma^2}$ of the classical two-dimensional Gaussian multiplicative chaos measure, this moment bound could go below $1$ with $\gamma\in(0,2)$.

Notice also the above moment bound seems to be coming from a one-dimensional scaling. This is indeed reflected in the proof of this result in~\cite{Huang:2023aa}, which suggests that one should not follow verbatim the original localization scheme of~\cite{Rhodes_2019}, since the bulk measure we consider is highly non-homogeneous in space and is concentrated around the boundary. Namely, comparing the two different moment bounds above, we know the contribution from the region with positive distance from the boundary should not be reflected in the right tail of the bulk measure, see also Remark~\ref{rema:only_boundary_covariance_matters}. Combining with considerations coming from the boundary Liouville conformal field theory, this suggests that we should introduce in parallel the boundary measure
\begin{equation*}
    \mu^{\partial}(I)=\lim_{\epsilon\to 0}\int_{I}e^{\frac{\gamma}{2}X_\epsilon(w)-\frac{\gamma^2}{8}\mathbb{E}[X_{\epsilon}(w)^2]}dw,\quad I\subset\mathbb{R},
\end{equation*}
consider the joint law of the bulk/boundary measures $(\mu,\mu^{\partial})$, and perform the localization trick only with respect to this boundary measure $\mu^{\partial}$ (instead of the bulk measure $\mu^{\mathrm{H}}$). This turns out to be the correct way of effectively pinning down the mass of $\mu^{\mathrm{H}}(Q_r)$ and to study its right tail.

More precisely, we apply the following localization scheme at the boundary, i.e. by tilting with respect to the above boundary Gaussian multiplicative chaos measure $\mu^{\partial}$:
\begin{equation*}
    \mathbb{P}[\mu(Q_r)>t]=\mathbb{E}\left[\frac{\mu^{\partial}(I_r)}{\mu^{\partial}(I_r)}\mathbf{1}_{\{\mu^{\mathrm{H}}(Q_r)>t\}}\right]=\int_{-r}^{r}\mathbb{E}\left[\frac{\mathbf{1}_{\{\mu^{\mathrm{H}}_{v}(Q_r)>t\}}}{\mu^{\partial}_{v}(I_r)}\right]dv,
\end{equation*}
where we used the Girsanov transformation with the tilting term $\mu^{\partial}(I_r)$, and the localized measures (at $v\in[-r,r]\subset\mathbb{R}$) are defined respectively as
\begin{equation}\label{eq:LocalizedBulkMeasure}
    \mu^{\mathrm{H}}_{v}(A)=\int_{A}|z-v|^{-\gamma^2}\mu^{\mathrm{H}}(d^2z),\quad A\subset Q_r,
\end{equation}
and
\begin{equation}\label{eq:LocalizedBoundaryMeasure}
    \mu^{\partial}_{v}(I)=\int_{I}|w-v|^{-\frac{\gamma^2}{2}}\mu^{\partial}(dw),\quad I\subset\mathbb{R}.
\end{equation}

\subsection{Decomposition of the Gaussian multiplicative chaos measure around the singularity}\label{subse:decomposition_of_the_gaussian_multiplicative_chaos_measure_around_the_singularity}

Using the Williams' decomposition for drifted Brownian motions (see Section~\ref{subse:williams_decomposition}), we can lift the above polar decomposition of the log-correlated Gaussian field $X$ to a decomposition of the localized bulk/boundary Gaussian multiplicative chaos measures~\eqref{eq:LocalizedBulkMeasure} and~\eqref{eq:LocalizedBoundaryMeasure} around the singularity $v\in\mathbb{R}$, which we take to be $0$ for simplicity.

Concretely, we perform the same decomposition as in Section~\ref{subse:semi_circle_polar_decomposition}, but restricted to the half-ball $B(0,\rho)$. This can be obtained simply by scaling, but we follow the notations of~\cite[Section~3.4]{Rhodes_2019} for the convenience of the reader. We introduce the average of $X$ on the semi-circle $S(0,\rho)$:
\begin{equation*}
    N_{\rho}=\frac{1}{\pi}\int_{0}^{\pi}X(\rho e^{i\theta})d\theta.
\end{equation*}
By a standard calculation, $\mathbb{E}[N_{\rho}^2]=-2\ln \rho$. Using the radial decomposition of Section~\ref{subse:semi_circle_polar_decomposition}, we have the following equalities in distribution:
\begin{equation}\label{eq:RadialRepresentations}
\begin{split}
    \mu^{\mathrm{H}}_0(Q(0,\rho))&=\rho^{2-\frac{\gamma^2}{2}}e^{\gamma N_{\rho}}e^{\gamma M}I^{\mathrm{H}}_\gamma(M),\\
    \mu^{\partial}_0(I(0,\rho))&=\rho^{1-\frac{\gamma^2}{4}}e^{\frac{\gamma}{2}N_{\rho}}e^{\frac{\gamma}{2}M}I^{\partial}_{\gamma}(M),
\end{split}
\end{equation}
where $B^{\gamma}_s$ denotes the two-sided drifted Brownian motion $\sqrt{2}B_s+(\frac{\gamma}{2}-\frac{2}{\gamma})s$ conditioned to be negative, and we set
\begin{equation*}
    I^{\mathrm{H}}_{\gamma}(M)=\int_{-L_{-M}}^{\infty}e^{\gamma B^{\gamma}_s}Z^{\mathrm{H}}_sds,\quad I^{\partial}_{\gamma}(M)=\int_{-L_{-M}}^{\infty}e^{\frac{\gamma}{2} B^{\gamma}_s}Z^{\partial}_sds.
\end{equation*}
In the sequel we sometimes drop the subscript/superscript $\gamma$ for readability, and write simply $I^{\mathrm{H}}(\star)$ or $I^{\partial}(\star)$ where $\star$ can be any real number using the above formulas.

\begin{rema}\label{rema:TrivialMonotonicity}
Notice that the both functions $x\mapsto I^{\mathrm{H}}(x)$ and $x\mapsto I^{\partial}(x)$ are increasing in $x$, by the positivity of the integrands. We also denote by
\begin{equation}\label{eq:Definitions_I}
    I^{\mathrm{H}}(\infty)=\int_{-\infty}^{\infty}e^{\gamma B^{\gamma}_s}Z^{\mathrm{H}}_sds,\quad I^{\partial}(\infty)=\int_{-\infty}^{\infty}e^{\frac{\gamma}{2}B^{\gamma}_s}Z^{\partial}_sds,
\end{equation}

These last two Gaussian multiplicative chaos measures are the ones appearing in the expression of the tail profile constant of Theorem~\ref{th:main_result}. Notice that their quotient appearing in Theorem~\ref{th:main_result} (as well as the expectation thereof)
\begin{equation*}
    \frac{I^{\mathrm{H}}(x)^{\frac{2}{\gamma^2}}}{I^{\partial}(x)}
\end{equation*}
has no reason to be monotonic in $x$.
\end{rema}

\subsection{Heuristics with drifted Brownian motions}\label{subse:heuristics_with_drifted_brownian_motions}
It turns out that the expression $\mathbb{E}\left[\frac{\mathbf{1}_{\{\mu^{\mathrm{H}}_{v}(Q_r)>t\}}}{\mu^{\partial}_{v}(I_r)}\right]$ is so-called ``local'' and does not differ too much for each $v\in[-r,r]$ (this point will be justified later in Section~\ref{sec:reduction_to_local_estimates}), we continue our illustration of the method with the special choice $v=0$ for simplicity.

We now present a heuristic which illustrates that the above localization trick at the boundary is the correct strategy, in the sense that it yields the correct decay rate $t^{-\frac{2}{\gamma^2}}$ in Theorem~\ref{th:main_result}.

Recall the radial decomposition of the log-correlated field around the point-wise singularity $v$ at the boundary of Section~\ref{subse:decomposition_of_the_gaussian_multiplicative_chaos_measure_around_the_singularity}. Now we forget about the ``lateral noise'' part denoted by $Y^{\bullet}$ (and therefore $Z^{\bullet}$), and only keep the drifted Brownian motion part $B^{\gamma}_s$ in the rest of this section.

More concretely, recall that we want to estimate
\begin{equation*}
    \mathbb{E}\left[\frac{\mathbf{1}_{\{\mu^{\mathrm{H}}_{0}(Q_r)>t\}}}{\mu^{\partial}_{0}(I_r)}\right].
\end{equation*}
Using the expressions~\eqref{eq:RadialRepresentations} but forgetting the lateral noises $Z^{\bullet}$, this is reduces to estimating
\begin{equation*}
    \mathbb{E}\left[\frac{\mathbf{1}_{\{r^{2-\frac{\gamma^2}{2}}e^{\gamma N_{r}}e^{\gamma M}\int_{0}^{\infty}e^{\gamma B_s^{\gamma}}ds>t\}}}{r^{1-\frac{\gamma^2}{4}}e^{\frac{\gamma}{2}N_r}e^{\frac{\gamma}{2}M}\int_{0}^{\infty}e^{\frac{\gamma}{2} B_s^{\gamma}}ds}\right]
\end{equation*}
where $B^{\gamma}_s$ has the same law as $\sqrt{2}B_s+(\frac{\gamma}{2}-\frac{2}{\gamma})s$ conditioned to stay below $0$, and $M$ is the maximum of the corresponding two-side drifted Brownian motion.

Now we use the independence of $M$ via Williams decomposition recalled in Section~\ref{subse:williams_decomposition} and write the above as
\begin{equation*}
    \mathbb{E}\left[\frac{\mathbf{1}_{\{e^{\gamma M}C^{\mathrm{H}}>t\}}}{e^{\frac{\gamma}{2}M}C^{\partial}}\right],
\end{equation*}
where $C^{\mathrm{H}}$ and $C^{\partial}$ are random constants independent of $M$. Now we can use the explicit distribution of the maximum $M$ recall in Section~\ref{subse:explicit_calculations_on_the_maximum_of_a_drifted_brownian_motion} to get the equality\footnote{Recall that if a positive random variable $X$ satisfies $\mathbb{P}[X>t]=t^{-\alpha}$, then $\mathbb{E}[X^{-1}\mathbf{1}_{\{X>t\}}]=\frac{\alpha}{\alpha+1}t^{-(\alpha+1)}$.}
\begin{equation*}
    \mathbb{E}\left[\frac{\mathbf{1}_{\{e^{\gamma M}C^{\mathrm{H}}>t\}}}{e^{\frac{\gamma}{2}M}C^{\partial}}\right]=(1-\frac{\gamma^2}{4})\mathbb{E}\left[\frac{(C^{\mathrm{H}})^{\frac{2}{\gamma^2}}}{C^{\partial}}\right]t^{-\frac{2}{\gamma^2}}.
\end{equation*}
The importance point here is that the power in $t$ in the above display is exactly as claimed in Theorem~\ref{th:main_result}. Putting back the expressions of $C^{\mathrm{H}}$ and $C^{\partial}$, we see that the right hand side above looks similar to the constant in Theorem~\ref{th:main_result} modulo the lateral noise parts.

The rest of this article is devoted to incorporate the lateral noise part in the above heuristics. This requires some careful control of the bulk/boundary quotients, since the noise parts in the bulk (e.g. $Z^{\mathrm{H}}$) and on the boundary (e.g. $Z^{\partial}$) interact in a highly non-trivial way. Indeed, as we have announced in Theorem~\ref{th:main_result}, the reflection coefficient is expressed in terms of a variant of the bulk/boundary quotient, which existence will only be established later in Lemma~\ref{lemm:finiteness_of_the_tail_profile_constant} using results from~\cite{Huang:2025aa}.

\section{Reduction to local estimates}\label{sec:reduction_to_local_estimates}
In this section, we reduce some key estimates to a local study near the singularity created by the localization at the boundary procedure. Recall that in Section~\ref{sec:strategy_of_proof_and_some_heuristics}, we showed that
\begin{equation}\label{eq:localized_integral}
    \mathbb{P}[\mu^{\mathrm{H}}(Q_r)>t]=\int_{v=-r}^{r}\mathbb{E}\left[\frac{\mathbf{1}_{\{\mu^{\mathrm{H}}_v(Q_r)>t\}}}{\mu^{\partial}_v(I_r)}\right]dv
\end{equation}
using the localization trick at the boundary, where $\mu^{\mathrm{H}}_v$, $\mu^{\partial}_v$ are respectively the bulk and boundary Gaussian multiplicative chaos measure localized at $v\in[-r,r]$ defined in~\eqref{eq:LocalizedBulkMeasure} and~\eqref{eq:LocalizedBoundaryMeasure}.

The goal of this section is to reduce the study of the integrand in the above display to a local estimate near the inserted point $v\in\mathbb{R}$ by establishing the following lemma:
\begin{lemm}[Getting rid of the non-singularity]\label{lemm:getting_rid_of_the_non-singularity}
Consider
\begin{equation*}
    Q(v,\rho)=B(v,\rho)\cap\overline{\mathbb{H}},\quad I(v,\rho)=[v-\rho,v+\rho]\subset\mathbb{R}
\end{equation*}
respectively the half ball in the upper half-plane centered at $v$ and the real interval centered at $v$, both of radius $\rho>0$. For each $v\in [-r,r]$, we can replace
\begin{equation*}
    \mathbb{E}\left[\frac{\mathbf{1}_{\{\mu^{\mathrm{H}}_v(Q_r)>t\}}}{\mu^{\partial}_v(I_r)}\right]
\end{equation*}
by a local study near the singularity point $v$ of the form
\begin{equation*}
    \mathbb{E}\left[\frac{\mathbf{1}_{\{\mu^{\mathrm{H}}_v(Q(v,\rho))>t\}}}{\mu^{\partial}_v(I(v,\rho))}\right]
\end{equation*}
within an error term of order $\rho^{-\kappa}o(t^{-\frac{2}{\gamma^2}-\delta})$ for some fixed $\delta>0$ and $\kappa\in(0,1)$ as long as $2\rho<\min(r-v,v+r)$ so that $Q(v,2\rho)\subset Q_r$ and $I(v,2\rho)\subset I_r$.
\end{lemm}
Under the conclusions of this lemma, we can replace each integrand $\mathbb{E}\left[\frac{\mathbf{1}_{\{\mu^{\mathrm{H}}_v(Q_r)>t\}}}{\mu^{\partial}_v(I_r)}\right]$ by its local version
\begin{equation*}
    \mathbb{E}\left[\frac{\mathbf{1}_{\{\mu^{\mathrm{H}}_v(Q(v,\rho))>t\}}}{\mu^{\partial}_v(I(v,\rho))}\right]
\end{equation*}
with negligeable error which is integrable over $v\in[-r,r]$ since $\kappa\in(0,1)$ and is $o(t^{-\frac{2}{\gamma^2}-\delta})$ with some $\delta>0$, and thus confirming our heuristics that ``the integrand is roughly the same for each $v\in[-r,r]$'' in a rigorous way. Indeed, we can further suppose that $v=0$ in the above display by translation invariance: estimates in this case will then be written down carefully in Section~\ref{sec:expression_of_the_universal_constant}.

\subsection{Getting rid of the non-singularity: overview of the strategy}\label{subse:getting_rid_of_the_non-singularity_overview_of_the_strategy}
The procedure of throwing out Gaussian multiplicative chaos masses outside of $B(v,\rho)$ is called in both~\cite[Section~3.5]{Rhodes_2019} and~\cite[Lemma~12]{wong2020universal} as ``getting rid of the non-singularity''. The situation is slightly more complicated in the setting of this paper because of the boundary term (in the sense that we should carefully discard boundary measures outside of $I(v,\rho)$ simultaneously with the bulk measure), as the bulk/boundary quotient is no longer a measure and does not enjoy additivity or monotonicity properties anymore. We roughly follow the notations of~\cite{Rhodes_2019} for the rest of this section.

Assume $2\rho<\min(r-v,v+r)$ so that $Q(v,2\rho)\subset Q_r$ and $I(v,2\rho)\subset I_r$ in the rest of this section. We denote by
\begin{equation*}
    Q(v,\rho)^{c}=Q_r\setminus Q(v,\rho),\quad I(v,\rho)^{c}=I_r\setminus I(v,\rho)
\end{equation*}
the complements of $Q(v,\rho)$ and $I(v,\rho)$ of the original bulk/boundary regions, respectively. We now establish two bounds, that both
\begin{equation*}
    \mathbb{E}\left[\frac{\mathbf{1}_{\{\mu^{\mathrm{H}}_v(Q_r)>t\}}}{\mu^{\partial}_v(I_r)}\right]-\mathbb{E}\left[\frac{\mathbf{1}_{\{\mu^{\mathrm{H}}_v(Q(v,\rho))>t\}}}{\mu^{\partial}_0(I(v,\rho))}\right]
\end{equation*}
and
\begin{equation*}
    \mathbb{E}\left[\frac{\mathbf{1}_{\{\mu^{\mathrm{H}}_v(Q(v,\rho))>t\}}}{\mu^{\partial}_v(I(v,\rho))}\right]-\mathbb{E}\left[\frac{\mathbf{1}_{\{\mu^{\mathrm{H}}_v(Q_r)>t\}}}{\mu^{\partial}_v(I_r)}\right]
\end{equation*}
are of order $\rho^{-\kappa}o(t^{-\frac{2}{\gamma^2}-\delta})$ with some $\kappa\in(0,1)$ and $\delta>0$. Before we get into the details of the proof, let us explain some guiding heuristics for this approximation.

The heuristic picture is the following. The indicator event $\mathbf{1}_{\{\mu^{\mathrm{H}}_v(Q_r)>t\}}$ is a large deviation type condition, and one can think of this as saying that $\mu^{\mathrm{H}}_v(Q_r)$ is of order $t$. Since this measure is localized at $0$, it received a strong singularity at $0$, and most of its mass comes from $\mu^{\mathrm{H}}_v(Q(v,\rho))$. In this way, the contribution from the complement, $\mu^{\mathrm{H}}_v(Q(v,\rho)^{c})$ should be of order $o(t)$. When this is true, the Brownian motion decomposition heuristics from Section~\ref{subse:heuristics_with_drifted_brownian_motions} shows that the corresponding boundary mass $\mu^{\partial}_v(I(v,\rho))$ should be of order $t^{\frac{1}{2}}$. Correspondingly, the complementary boundary mass $\mu^{\partial}_v(I(v,\rho)^{c})$ should be much smaller, of order $o(t^{\frac{1}{2}})$. Therefore, we should use cut-offs at scale $t$ for the bulk measure part in the nominator, and at scale $t^{\frac{1}{2}}$ for the boundary measure part in the denominator. We will see by explicit calculations that these are effectively the correct cut-off scales, and we will use moment estimates below to justify these intuitions.

\subsection{An upper bound}\label{subse:an_upper_bound}
We first upper bound $\mathbb{E}\left[\frac{\mathbf{1}_{\{\mu^{\mathrm{H}}_v(Q_r)>t\}}}{\mu^{\partial}_v(I_r)}\right]$ by $\mathbb{E}\left[\frac{\mathbf{1}_{\{\mu^{\mathrm{H}}_v(Q(v,\rho))>t\}}}{\mu^{\partial}_v(I(v,\rho))}\right]$ up to some correction term of order $o(t^{-\frac{2}{\gamma^2}})$ and appropriate dependence in $\rho$. To this end, let $\eta\in(0,1)$ be some cut-off parameter that we choose later, and write
\begin{equation*}
\begin{split}
    \mathbb{E}\left[\frac{\mathbf{1}_{\{\mu^{\mathrm{H}}_v(Q_r)>t\}}}{\mu^{\partial}_v(I_r)}\right]&=\mathbb{E}\left[\frac{\mathbf{1}_{\{\mu^{\mathrm{H}}_v(Q(v,\rho))+\mu^{\mathrm{H}}_v(Q(v,\rho)^{c})>t\}}}{\mu^{\partial}_v(I(v,\rho))+\mu^{\partial}_v(I(v,\rho)^{c})}\right]\\
    &\leq\mathbb{E}\left[\frac{\mathbf{1}_{\{\mu^{\mathrm{H}}_v(Q(v,\rho))>t-t^{1-\eta}\}}+\mathbf{1}_{\{\mu^{\mathrm{H}}_v(Q(v,\rho)^{c})>t^{1-\eta}\}}}{\mu^{\partial}_v(I(v,\rho))+\mu^{\partial}_v(I(v,\rho)^{c})}\right]\\
    &\leq\mathbb{E}\left[\frac{\mathbf{1}_{\{\mu^{\mathrm{H}}_v(Q(v,\rho))>t-t^{1-\eta}\}}}{\mu^{\partial}_v(I(v,\rho))}\right]+\mathbb{E}\left[\frac{\mathbf{1}_{\{\mu^{\mathrm{H}}_v(Q(v,\rho)^{c})>t^{1-\eta}\}}}{\mu^{\partial}_v(I_r)}\right].
\end{split}
\end{equation*}
We now indicate how we treat the two terms above:
\begin{enumerate}
    \item The first term will be studied in Section~\ref{sec:expression_of_the_universal_constant}, for which we will establish an estimate of the form with some $\kappa\in(0,1)$ and $\delta>0$:
    \begin{equation*}
        \mathbb{E}\left[\frac{\mathbf{1}_{\{\mu^{\mathrm{H}}_v(Q(v,\rho))>t\}}}{\mu^{\partial}_v(I(v,\rho))}\right]=Ct^{-\frac{2}{\gamma^2}}+\rho^{-\kappa}o(t^{-\frac{2}{\gamma^2}-\delta}).
    \end{equation*}
    Changing $t$ to $t-t^{1-\eta}$ on the left hand side with $\eta\in(0,1)$ will not affect this estimate (in particular, with the same constant $C$ for the right tail coefficient).
    \item The second term is $o(t^{-\frac{2}{\gamma^2}-\delta})$ since for any $p>0$,
\begin{equation*}
    \mathbb{E}\left[\frac{\mathbf{1}_{\{\mu^{\mathrm{H}}_v(Q(v,\rho)^{c})>t^{1-\eta}\}}}{\mu^{\partial}_v(I_r)}\right]\leq t^{-p(1-\eta)}\mathbb{E}\left[\frac{\mu^{\mathrm{H}}_v(Q(v,\rho)^{c})^{p}}{\mu^{\partial}_v(I_r)}\right],
\end{equation*}
and the last expectation term is smaller than $C\rho^{-\kappa}$ with some $\kappa\in(0,1)$ for $p$ close enough to $\min(\frac{2}{\gamma^2}+\frac{1}{2},\frac{4}{\gamma^2})$ from below, see Remark~\ref{rema:FinalLemma}. In particular, we can choose $p>\frac{2}{\gamma^2}$ and $p(1-\eta)>\frac{2}{\gamma^2}$ such that the second term is also $\rho^{-\kappa}o(t^{-\frac{2}{\gamma^2}-\delta})$ for some $\kappa\in(0,1)$ and $\delta>0$.
\end{enumerate}
Combining these observations, we get the desired upper bound estimate up to an error of order $\rho^{-\kappa}o(t^{-\frac{2}{\gamma^2}-\delta})$ for some small but positive constant $\delta>0$ and some $\kappa\in(0,1)$. The condition $\kappa\in(0,1)$ ensures that~\eqref{eq:localized_integral} is integrable in the $v$ variable, thus we get the upper bound
\begin{equation*}
    \mathbb{P}[\mu^{\mathrm{H}}(Q_r)>t]=\int_{v=-r}^{r}\mathbb{E}\left[\frac{\mathbf{1}_{\{\mu^{\mathrm{H}}_v(Q_r)>t\}}}{\mu^{\partial}_v(I_r)}\right]dv\leq 2r\cdot \mathbb{E}\left[\frac{\mathbf{1}_{\{\mu^{\mathrm{H}}_0(Q_r)>t\}}}{\mu^{\partial}_0(I_r)}\right]+o(t^{-\frac{2}{\gamma^2}-\delta}).
\end{equation*}

\subsection{A lower bound}\label{subse:a_lower_bound}
We now show the converse to the previous section, that we lower bound $\mathbb{E}\left[\frac{\mathbf{1}_{\{\mu^{\mathrm{H}}_v(Q_r)>t\}}}{\mu^{\partial}_v(I_r)}\right]$ by $\mathbb{E}\left[\frac{\mathbf{1}_{\{\mu^{\mathrm{H}}_v(Q(v,\rho))>t\}}}{\mu^{\partial}_v(I(v,\rho))}\right]$ up to some correction term of order $\rho^{-\kappa}o(t^{-\frac{2}{\gamma^2}-\delta})$ for some $\kappa\in(0,1)$ and $\delta>0$. The condition $\kappa\in(0,1)$ insures the integrability of the $v$ variable in~\eqref{eq:localized_integral} as discussed in the previous paragraph.

To this end, let $\eta'\in(0,1)$ be some cut-off parameter that we choose later, and write
\begin{equation*}
\begin{split}
    \mathbb{E}\left[\frac{\mathbf{1}_{\{\mu^{\mathrm{H}}_v(Q_r)>t\}}}{\mu^{\partial}_v(I_r)}\right]&\geq\mathbb{E}\left[\frac{\mathbf{1}_{\{\mu^{\mathrm{H}}_v(Q(v,\rho))>t\}}}{\mu^{\partial}_v(I_r)}\right]\\
    &\geq\mathbb{E}\left[\frac{\mathbf{1}_{\{\mu^{\mathrm{H}}_v(Q(v,\rho))>t\}}\mathbf{1}_{\{\mu^{\partial}_v(I(v,\rho)^{c})<t^{\frac{1}{2}-\eta'}\}}}{\mu^{\partial}_v(I_r)}\right]\\
    &\geq\mathbb{E}\left[\frac{\mathbf{1}_{\{\mu^{\mathrm{H}}_v(Q(v,\rho))>t\}}\mathbf{1}_{\{\mu^{\partial}_v(I(v,\rho)^{c})<t^{\frac{1}{2}-\eta'}\}}}{\mu^{\partial}_v(I(v,\rho))+t^{\frac{1}{2}-\eta'}}\right]\\
    &\geq \mathbb{E}\left[\frac{\mathbf{1}_{\{\mu^{\mathrm{H}}_v(Q(v,\rho))>t\}}}{\mu^{\partial}_v(I(v,\rho))+t^{\frac{1}{2}-\eta'}}\right]-\mathbb{E}\left[\frac{\mathbf{1}_{\{\mu^{\mathrm{H}}_v(Q(v,\rho))>t\}}\mathbf{1}_{\{\mu^{\partial}_v(I(v,\rho)^{c})>t^{\frac{1}{2}-\eta'}\}}}{\mu^{\partial}_v(I(v,\rho))+t^{\frac{1}{2}-\eta'}}\right]\\
    &\geq \mathbb{E}\left[\frac{\mathbf{1}_{\{\mu^{\mathrm{H}}_v(Q(v,\rho))>t\}}}{\mu^{\partial}_v(I(v,\rho))}\right]-t^{-(\frac{1}{2}-\eta')}\mathbb{E}\left[\frac{\mathbf{1}_{\{\mu^{\mathrm{H}}_v(Q(v,\rho))>t\}}}{\mu^{\partial}_v(I(v,\rho))^2}\right]-t^{-(\frac{1}{2}-\eta')}\mathbb{P}\left[\mu^{\partial}_v(I(v,\rho)^{c})>t^{\frac{1}{2}-\eta'}\right]
\end{split}
\end{equation*}
where the elementary inequality $(1+u)^{-1}\geq 1-u$ for $u>0$ is used at the last inequality. The last two terms are the error terms:
\begin{enumerate}
    \item For the second term, notice that for any $p>0$,
    \begin{equation*}
        t^{-(\frac{1}{2}-\eta')}\mathbb{E}\left[\frac{\mathbf{1}_{\{\mu^{\mathrm{H}}_v(Q(v,\rho))>t\}}}{\mu^{\partial}_v(I(v,\rho))^2}\right]\leq t^{-(p+\frac{1}{2}-\eta')}\mathbb{E}\left[\frac{\mu^{\mathrm{H}}_v(Q(v,\rho))^{p}}{\mu^{\partial}_v(I(v,\rho))^2}\right].
    \end{equation*}
    Remark~\ref{rema:FinalLemma} shows that the last expectation term is smaller than $C\rho^{-\kappa}$ with some $\kappa\in(0,1)$ for $p=\frac{2}{\gamma^2}$ and any $\eta'\in(0,\frac{1}{2})$.
    \item The third term is also $\rho^{-\kappa}o(t^{-\frac{2}{\gamma^2}-\delta})$ for small enough $\eta'$, since for any $p>0$,
    \begin{equation*}
        t^{-(\frac{1}{2}-\eta')}\mathbb{P}\left[\mu^{\partial}_v(I(v,\rho)^{c})>t^{\frac{1}{2}-\eta'}\right]\leq t^{-(p+1)(\frac{1}{2}-\eta')}\mathbb{E}[\mu^{\partial}_v(I(v,\rho)^{c})^{p}].
    \end{equation*}
    Now that we have taken out the singularity $I(0,\rho)$, the mass $\mu^{\partial}_0(I(0,\rho)^{c})$ is comparable to that of a classical one-dimensional Gaussian multiplicative chaos on the boundary, modulo some multiplicative constant that only depends on $\rho$. Recall that the moment bound for the auxiliary boundary Gaussian multiplicative chaos measure $\mu^{\partial}$ (which is defined with $-2\ln|z-w|$ covariance but $\frac{\gamma}{2}$ coupling constant) is $p<\frac{4}{\gamma^2}$. Especially, taking $p$ close to this threshold from below,~\cite[Lemma~3.6]{Rhodes_2019} yields that the third term is of order
    \begin{equation*}
        C_{p}\rho^{-\kappa}t^{-(\frac{4}{\gamma^2}+1)^{-}(\frac{1}{2}-\eta')}=\rho^{-\kappa}o(t^{-\frac{2}{\gamma^2}-\delta})
    \end{equation*}
    with some $\kappa\in(0,1)$ and small enough $\eta'\in(0,\frac{1}{2})$.
\end{enumerate}
Combining these observations, we get the desired lower bound estimate
\begin{equation*}
    \mathbb{P}[\mu^{\mathrm{H}}(Q_r)>t]=\int_{v=-r}^{r}\mathbb{E}\left[\frac{\mathbf{1}_{\{\mu^{\mathrm{H}}_v(Q_r)>t\}}}{\mu^{\partial}_v(I_r)}\right]dv\geq 2r\cdot \mathbb{E}\left[\frac{\mathbf{1}_{\{\mu^{\mathrm{H}}_0(Q_r)>t\}}}{\mu^{\partial}_0(I_r)}\right]+o(t^{-\frac{2}{\gamma^2}-\delta}).
\end{equation*}

\section{Expression of the universal tail profile constant}\label{sec:expression_of_the_universal_constant}
We now proceed to the proof of our main Theorem~\ref{th:main_result}, and derive the Gaussian multiplicative chaos expression for the universal tail constant. Our argument is based on~\cite{Rhodes_2019}, although we need the non-trivial preliminary estimates in the companion paper~\cite{Huang:2025aa} to adapt their arguments to the boundary case. More precisely, Lemma~\ref{lemm:finiteness_of_the_tail_profile_constant} below, which can be seen as a variant of~\cite[Theorem~2]{Huang:2025aa}, is the bulk/boundary quotient analog of~\cite[Lemma~3.3]{Kupiainen_2020} (or~\cite[Equation~(1.11)]{Rhodes_2019}), and the lower/upper bound estimates are done similarly to~\cite{Rhodes_2019} provided this preliminary lemma. Suitable technical changes in the calculations due to the higher-dimensional uniform boundary singularity will be explained in a streamlined fashion.

According to Lemma~\ref{lemm:getting_rid_of_the_non-singularity} of Section~\ref{sec:reduction_to_local_estimates}, it remains to study the behavior near the singularity, i.e.
\begin{equation*}
    \mathbb{E}\left[\frac{\mathbf{1}_{\{\mu^{\mathrm{H}}_v(Q(v,\rho))>t\}}}{\mu^{\partial}_v(I(v,\rho))}\right].
\end{equation*}
By translation invariance of the covariance kernel~\eqref{eq:ExactScalingNeumannKernel}, we will take $v=0$ until the end of this section.

\subsection{Some preparations}\label{subse:some_preparations}
Recall the modified Gaussian multiplicative chaos masses $I^{\mathrm{H}}(\infty)$ and $I^{\partial}(\infty)$ defined in Remark~\ref{rema:TrivialMonotonicity} and which are used to define the tail profile constant $C_{(1,0)}$ in Theorem~\ref{th:main_result}. In this section, we establish a moment bound of type
\begin{equation*}
    \mathbb{E}\left[\frac{I^{\mathrm{H}}(\infty)^{\frac{2}{\gamma^2}}}{I^{\partial}(\infty)}\right]<\infty.
\end{equation*}
In fact, from the definitions of $I^{\mathrm{H}}$ and $I^{\partial}$, the above expectation term is almost the same as the joint moment of the bulk/boundary quotient of Gaussian multiplicative chaos measures studied in~\cite[Theorem~1]{Huang:2025aa}, but with the maximum of the Brownian motion removed as a global constant in both the nominator and the denominator. Therefore, it is not surprising that the proof relies quite heavily on the preliminary results of~\cite{Huang:2025aa}. It is useful for future propose to record the following lemma for general $(p,q)$-quotients.
\begin{lemm}[Finiteness of the tail profile constant]\label{lemm:finiteness_of_the_tail_profile_constant}
For any $q>0$ and $0\leq p<\min(\frac{2}{\gamma^2}+\frac{q}{2},\frac{4}{\gamma^2})$,
\begin{equation*}
    \mathbb{E}\left[\frac{I^{\mathrm{H}}(\infty)^{p}}{I^{\partial}(\infty)^{q}}\right]<\infty.
\end{equation*}
\end{lemm}

We now briefly explain the strategy of the proof below, since the idea behind this proof is related to the main theorem of~\cite{Huang:2025aa}. The main heuristic behind~\cite[Theorem~2]{Huang:2025aa}, quickly resumed, is that ``the boundary measure scales like the square root of the bulk measure'', and we needed to separate the cases with $q<2p$ and $q\geq 2p$ since the bulk/boundary quotient behaves like a positive (resp. negative) moment of Gaussian multiplicative chaos, implying quite different steps of proofs. Here the phenomenon is similar (with the Brownian motion in the radial decomposition separated out), hence the distinction of cases below.

\begin{proof}
Recall first that when $p=0$, Lemma~\ref{lemm:finiteness_of_the_tail_profile_constant} is included in~\cite[Lemma~3.3]{Kupiainen_2020} and is shown in~\cite[Section~12.3]{Kupiainen_2020}. We will refer to these lemmas in the other cases.

$\bullet$ Suppose now $p>0$ and $q<2p$. We will show the following, that both
\begin{equation}\label{eq:LittleScaling_n_to_n+1}
    \mathbb{E}\left[\frac{(\int_{n}^{n+1}e^{\gamma B^{\gamma}_s}Z^{\mathrm{H}}_sds)^{p}}{(\int_{n}^{n+1}e^{\frac{\gamma}{2}B^{\gamma}_s}Z^{\partial}_s ds)^{q}}\right]\quad\text{and}\quad\mathbb{E}\left[\frac{(\int_{-n-1}^{-n}e^{\gamma B^{\gamma}_s}Z^{\mathrm{H}}_sds)^{p}}{(\int_{-n-1}^{-n}e^{\frac{\gamma}{2}B^{\gamma}_s}Z^{\partial}_sds)^{q}}\right]
\end{equation}
are exponentially decaying in $n$. We now treat the first expectation as the second one follows by symmetry.

We first write $B^{\gamma}_s=B^{\gamma}_n+\widetilde{B}_{s-n}^{\gamma}$ for $s\in[n,n+1]$, where $\widetilde{B}_{s-n}^{\gamma}=B_s^{\gamma}-B_n^{\gamma}$. Apply the strong Markov property at time $n$ to write
\begin{equation*}
    \mathbb{E}\left[\frac{(\int_{n}^{n+1}e^{\gamma B^{\gamma}_s}Z^{\mathrm{H}}_sds)^{p}}{(\int_{n}^{n+1}e^{\frac{\gamma}{2}B^{\gamma}_s}Z^{\partial}_s ds)^{q}}\right]=\mathbb{E}\left[(e^{\gamma B^{\gamma}_n})^{p-\frac{q}{2}}\right]\mathbb{E}\left[\frac{(\int_{n}^{n+1}e^{\gamma \widetilde{B}^{\gamma}_{s-n}}Z^{\mathrm{H}}_sds)^{p}}{(\int_{n}^{n+1}e^{\frac{\gamma}{2}\widetilde{B}^{\gamma}_{s-n}}Z^{\partial}_s ds)^{q}} \Big\mid B^{\gamma}_n\right]
\end{equation*}
Using the independence of the joint process $(Z^{\mathrm{H}}_s,Z^{\partial}_s)$ with the Brownian motions, we can bound the last expectation by
\begin{equation*}
    \mathbb{E}\left[\frac{\left(e^{\gamma\sup_{s\in[n,n+1]}\widetilde{B}^{\gamma}_{s-n}}\right)^{p}}{\left(e^{\frac{\gamma}{2}\inf_{s\in[n,n+1]}\widetilde{B}^{\gamma}_{s-n}}\right)^{q}}\Big\mid B^{\gamma}_n\right]\mathbb{E}\left[\frac{(\int_{0}^{1}Z^{\mathrm{H}}_sds)^{p}}{(\int_{0}^{1}Z^{\partial}_s ds)^{q}}\right]
\end{equation*}
where we used the translation invariance of the joint process $(Z^{\mathrm{H}}_s,Z^{\partial}_s)$. By~\cite[Section~12.3]{Kupiainen_2020}, we have, using the first part in the proof of~\cite[Lemma~3.3]{Kupiainen_2020}
\begin{equation*}
\begin{split}
    \mathbb{E}\left[(e^{\gamma B^{\gamma}_n})^{p-\frac{q}{2}}\right]&\leq\mathbb{E}\left[\mathbf{1}_{\{B_n-\frac{1}{\sqrt{2}}(\frac{2}{\gamma}-\frac{\gamma}{2})n\leq\sqrt{2}\}}e^{\sqrt{2}\gamma(p-\frac{q}{2})(B_n-\frac{1}{\sqrt{2}}(\frac{2}{\gamma}-\frac{\gamma}{2})n)}\right]\\
    &=Cn^{-\frac{1}{2}}\int_{-\infty}^{1}e^{\sqrt{2}\gamma(p-\frac{q}{2})y}e^{-\frac{(y+\frac{1}{\sqrt{2}}(\frac{2}{\gamma}-\frac{\gamma}{2})n)^2}{2n}}dy\\
    &\leq Ce^{-\epsilon n}.
\end{split}
\end{equation*}
Indeed, the first inequality in the above display comes from~\cite[Section~12.4]{Kupiainen_2020}, that $\frac{1}{\sqrt{2}}B^{\gamma}_s$ is stochastically dominated by a drifted Brownian motion with drift $-\frac{1}{\sqrt{2}}(\frac{2}{\gamma}-\frac{\gamma}{2})$ conditionned to stay below $1$.

Then by Cauchy-Schwarz and proceed again as in the proof of~\cite[Lemma~3.3]{Kupiainen_2020},
\begin{equation}\label{eq:eq4coro}
    \mathbb{E}\left[\frac{\left(e^{\gamma\sup_{s\in[n,n+1]}\widetilde{B}^{\gamma}_{s-n}}\right)^{p}}{\left(e^{\frac{\gamma}{2}\inf_{s\in[n,n+1]}\widetilde{B}^{\gamma}_{s-n}}\right)^{q}}\Big\mid B^{\gamma}_n\right]\leq\mathbb{E}\left[\left(e^{\gamma\sup_{s\in[n,n+1]}\widetilde{B}^{\gamma}_{s-n}}\right)^{2p}\Big\mid B^{\gamma}_n\right]\mathbb{E}\left[\left(e^{\frac{\gamma}{2}\inf_{s\in[n,n+1]}\widetilde{B}^{\gamma}_{s-n}}\right)^{-2q}\Big\mid B^{\gamma}_n\right]<\infty.
\end{equation}
Combining all the elements above yields the following upper bound:
\begin{equation*}
    \mathbb{E}\left[\frac{(\int_{n}^{n+1}e^{\gamma B^{\gamma}_s}Z^{\mathrm{H}}_sds)^{p}}{(\int_{n}^{n+1}e^{\frac{\gamma}{2} B^{\gamma}_s}Z^{\partial}_s ds)^{q}}\right]\leq Ce^{-\epsilon n}\mathbb{E}\left[\frac{(\int_{0}^{1}Z^{\mathrm{H}}_sds)^{p}}{(\int_{0}^{1}Z^{\partial}_s ds)^{q}}\right],
\end{equation*}
where the finiteness of the last expectation a direct consequence of~\cite[Theorem~2]{Huang:2025aa}. Thus we have established the exponential decay of~\eqref{eq:LittleScaling_n_to_n+1}.

From the exponential decay of~\eqref{eq:LittleScaling_n_to_n+1} follows Lemma~\ref{lemm:finiteness_of_the_tail_profile_constant}. Indeed,
\begin{enumerate}
    \item If $0<p\leq 1$, then the subadditivity inequality yields
    \begin{equation*}
        \mathbb{E}\left[\frac{I^{\mathrm{H}}(\infty)^{p}}{I^{\partial}(\infty)^{q}}\right]\leq\sum_{n\in\mathbb{Z}}\mathbb{E}\left[\frac{(\int_{n}^{n+1}e^{\gamma B^{\gamma}_s}Z^{\mathrm{H}}_sds)^{p}}{I^{\partial}(\infty)^{q}}\right] \leq \sum_{n\in\mathbb{Z}}\mathbb{E}\left[\frac{(\int_{n}^{n+1}e^{\gamma B^{\gamma}_s}Z^{\mathrm{H}}_sds)^{p}}{(\int_{n}^{n+1}e^{\frac{\gamma}{2} B^{\gamma}_s}Z^{\partial}_s ds)^{q}}\right]\leq\infty,
    \end{equation*}
    by the exponential decay property above.
    \item If $p>1$, then Minkowski's inequality for the $L_p$-norm yields similar conclusion:
    \begin{equation*}
        \mathbb{E}\left[\frac{I^{\mathrm{H}}(\infty)^{p}}{I^{\partial}(\infty)^{q}}\right]^{\frac{1}{p}}\leq\sum_{n\in\mathbb{Z}}\mathbb{E}\left[\frac{(\int_{n}^{n+1}e^{\gamma B^{\gamma}_s}Z^{\mathrm{H}}_sds)^{p}}{I^{\partial}(\infty)^{q}}\right]^{\frac{1}{p}} \leq \sum_{n\in\mathbb{Z}}\mathbb{E}\left[\frac{(\int_{n}^{n+1}e^{\gamma B^{\gamma}_s}Z^{\mathrm{H}}_sds)^{p}}{(\int_{n}^{n+1}e^{\frac{\gamma}{2} B^{\gamma}_s}Z^{\partial}_s ds)^{q}}\right]^{\frac{1}{p}}\leq\infty.
    \end{equation*}
\end{enumerate}
This finishes the proof of Lemma~\ref{lemm:finiteness_of_the_tail_profile_constant} in the case $p>0$ and $q<2p$.

$\bullet$ Suppose now that $p>0$ and $q\geq 2p$. It suffices to use the second part of the proof of~\cite[Lemma~3.3]{Kupiainen_2020}. Indeed, write $q=\alpha+(q-\alpha)$ with $\alpha$ slightly below $2p$ to be fixed in a moment, and use Hölder's inequality to write
\begin{equation*}
    \mathbb{E}\left[\frac{I^{\mathrm{H}}(\infty)^{p}}{I^{\partial}(\infty)^{q}}\right]\leq\mathbb{E}\left[\frac{I^{\mathrm{H}}(\infty)^{mp}}{I^{\partial}(\infty)^{m\alpha}}\right]^{\frac{1}{m}}\mathbb{E}\left[I^{\partial}(\infty)^{-m'(q-\alpha)}\right]^{\frac{1}{m'}},
\end{equation*}
with positive conjugate parameters $\frac{1}{m}+\frac{1}{m'}=1$ and $m$ slightly above $1$ to be fixed. The second expectation is finite by~\cite[Lemma~3.3]{Kupiainen_2020} for all $\alpha<q$. By the previous part of the proof, the first expectation is finite when $mp<\min(\frac{2}{\gamma^2}+\frac{m\alpha}{2},\frac{4}{\gamma^2})$ and $2mp>m\alpha$: one checks that this is possible when $m$ is slightly above $1$ and $\alpha$ is slightly below $2p$.
\end{proof}

In particular, with $(p,q)=(\frac{2}{\gamma^2},1)$ and $\gamma\in(0,2)$, we have shown that the constant $C_{(1,0)}$ in Theorem~\ref{th:main_result} is finite. Examining the proof above, we also get the following corollaries:
\begin{coro}\label{coro:B_B}
For any $q>0$ and $p<\min(\frac{2}{\gamma^2}+\frac{q}{2},\frac{4}{\gamma^2})$,
\begin{equation*}
    \mathbb{E}\left[\frac{(\int_{0}^{\infty}e^{\gamma B^{\gamma}_s}Z^{\mathrm{H}}_sds)^{p}}{(\int_{0}^{\infty}e^{\frac{\gamma}{2}B^{\gamma}_s}Z^{\partial}_s ds)^{q}}\right]<\infty.
\end{equation*}
\end{coro}
\begin{coro}\label{coro:supM}
For any $q>0$ and $p<\min(\frac{2}{\gamma^2}+\frac{q}{2},\frac{4}{\gamma^2})$,
\begin{equation*}
    \sup_{x>0}\mathbb{E}\left[\frac{I^{\mathrm{H}}(x)^{p}}{I^{\partial}(x)^{q}}\right]<\infty.
\end{equation*}
\end{coro}
Notice that the function $x\mapsto \mathbb{E}\left[\frac{I^{\mathrm{H}}(x)^{p}}{I^{\partial}(x)^{q}}\right]$ is not necessarily monotonic on $\mathbb{R}_{>0}$, so that this corollary requires some explanation.
\begin{proof}
When $p>\frac{q}{2}$, it suffices to replace the upper bound in the previous proof (say when $0<p\leq 1$)
\begin{equation*}
    \mathbb{E}\left[\frac{I^{\mathrm{H}}(\infty)^{p}}{I^{\partial}(\infty)^{q}}\right]\leq\sum_{n\in\mathbb{Z}}\mathbb{E}\left[\frac{(\int_{n}^{n+1}e^{\gamma B^{\gamma}_s}Z^{\mathrm{H}}_sds)^{p}}{(\int_{n}^{n+1}e^{\frac{\gamma}{2} B^{\gamma}_s}Z^{\partial}_s ds)^{q}}\right]
\end{equation*}
by
\begin{equation*}
    \mathbb{E}\left[\frac{I^{\mathrm{H}}(x)^{p}}{I^{\partial}(x)^{q}}\right]\leq\sum_{n\leq x-1}\mathbb{E}\left[\frac{(\int_{n}^{n+1}e^{\gamma B^{\gamma}_s}Z^{\mathrm{H}}_sds)^{p}}{(\int_{n}^{n+1}e^{\frac{\gamma}{2} B^{\gamma}_s}Z^{\partial}_s ds)^{q}}\right]+\mathbb{E}\left[\frac{(\int_{\lfloor x\rfloor}^{x}e^{\gamma B^{\gamma}_s}Z^{\mathrm{H}}_sds)^{p}}{(\int_{\lfloor x\rfloor}^{x}e^{\frac{\gamma}{2} B^{\gamma}_s}Z^{\partial}_s ds)^{q}}\right],
\end{equation*}
and the last term is uniformly bounded by the same argument as in~\eqref{eq:eq4coro}. The proof when $p\leq\frac{q}{2}$ then follows from the same Hölder inequality argument.
\end{proof}

With the radial decomposition~\eqref{eq:RadialRepresentations} and the preliminary Lemma~\ref{lemm:finiteness_of_the_tail_profile_constant}, we are now ready to study the localized expression
\begin{equation}\label{eq:LocalizedExpression_I}
    \mathbb{E}\left[\frac{\mathbf{1}_{\{\mu^{\mathrm{H}}_0(Q(0,\rho))>t\}}}{\mu^{\partial}_0(I(0,\rho))}\right]=\mathbb{E}\left[\frac{\mathbf{1}_{\{\rho^{2-\frac{\gamma^2}{2}}e^{\gamma N_{\rho}}e^{\gamma M}I^{\mathrm{H}}(M)>t\}}}{\rho^{1-\frac{\gamma^2}{4}}e^{\frac{\gamma}{2}N_{\rho}}e^{\frac{\gamma}{2}M}I^{\partial}(M)}\right].
\end{equation}

\subsection{Lower bound estimate}\label{subse:lower_bound_estimate}
We first establish a lower bound for~\eqref{eq:LocalizedExpression_I}. We are \emph{almost} able to use the exact formula for the maximum $M$, recalled in Section~\ref{subse:explicit_calculations_on_the_maximum_of_a_drifted_brownian_motion}, if we replace $I^{\mathrm{H}}(M)$ and $I^{\partial}(M)$ by respectively $I^{\mathrm{H}}(\infty)$ and $I^{\partial}(\infty)$. Indeed, one can check by a formal calculation similar to that in Section~\ref{subse:heuristics_with_drifted_brownian_motions} that this would give exactly the desired formula announced in Theorem~\ref{th:main_result}. We will now justify rigorously this replacement by restricting ourselves to well-chosen events.

The general scheme follows that of~\cite{Rhodes_2019} but we need several extra technical treatments due to the bulk/boundary correlations. We begin by lower bounding the above display by restricting ourselves to the event $\{M>\frac{\eta}{\gamma}\ln t\}$ for some $\eta>\delta$ such that:
\begin{equation}\label{eq:ChoiceOfParameters}
    \frac{2}{\gamma^2}<p<\frac{2}{\gamma^2}+\Delta p\quad\text{and}\quad \eta>\delta\quad\text{and}\quad p(1-\eta)>\frac{2}{\gamma^2}+\delta.
\end{equation}
Above, $\Delta p$ is some small positive quantity that we fix later: one should think of taking $p$ slightly above $\frac{2}{\gamma^2}$ when this condition is invoked. By Remark~\ref{rema:TrivialMonotonicity}, on this event one has $I^{\mathrm{H}}(M)\geq I^{\mathrm{H}}(\frac{\eta}{\gamma}\ln t)$. We have then
\begin{equation*}
\begin{split}
    \mathbb{E}\left[\frac{\mathbf{1}_{\{\rho^{2-\frac{\gamma^2}{2}}e^{\gamma N_{\rho}}e^{\gamma M}I^{\mathrm{H}}(M)>t\}}}{\rho^{1-\frac{\gamma^2}{4}}e^{\frac{\gamma}{2}N_{\rho}}e^{\frac{\gamma}{2}M}I^{\partial}(M)}\right]&\geq \mathbb{E}\left[\frac{\mathbf{1}_{\{\rho^{2-\frac{\gamma^2}{2}}e^{\gamma N_{\rho}}e^{\gamma M}I^{\mathrm{H}}(M)>t\}}\mathbf{1}_{\{\gamma M>\eta\ln t\}}}{\rho^{1-\frac{\gamma^2}{4}}e^{\frac{\gamma}{2}N_{\rho}}e^{\frac{\gamma}{2}M}I^{\partial}(M)}\right]\\
    &\geq\mathbb{E}\left[\frac{\mathbf{1}_{\{\rho^{2-\frac{\gamma^2}{2}}e^{\gamma N_{\rho}}e^{\gamma M}I^{\mathrm{H}}(\frac{\eta}{\gamma}\ln t)>t\}}\mathbf{1}_{\{\gamma M>\eta\ln t\}}}{\rho^{1-\frac{\gamma^2}{4}}e^{\frac{\gamma}{2}N_{\rho}}e^{\frac{\gamma}{2}M}I^{\partial}(\infty)}\right]\\
    &=\mathbb{E}\left[\frac{\mathbf{1}_{\{\rho^{2-\frac{\gamma^2}{2}}e^{\gamma N_{\rho}}e^{\gamma M}I^{\mathrm{H}}(\frac{\eta}{\gamma}\ln t)>t\}}}{\rho^{1-\frac{\gamma^2}{4}}e^{\frac{\gamma}{2}N_{\rho}}e^{\frac{\gamma}{2}M}I^{\partial}(\infty)}\right]\\
    &\quad-\mathbb{E}\left[\frac{\mathbf{1}_{\{\rho^{2-\frac{\gamma^2}{2}}e^{\gamma N_{\rho}}e^{\gamma M}I^{\mathrm{H}}(\frac{\eta}{\gamma}\ln t)>t\}}\mathbf{1}_{\{\gamma M<\eta\ln t\}}}{\rho^{1-\frac{\gamma^2}{4}}e^{\frac{\gamma}{2}N_{\rho}}e^{\frac{\gamma}{2}M}I^{\partial}(\infty)}\right].
\end{split}
\end{equation*}

\begin{enumerate}
    \item We first get rid of the second expectation term above by showing that it is of order $o(t^{-\frac{2}{\gamma^2}-\delta})$. Indeed, notice that
    \begin{equation*}
        \{\rho^{2-\frac{\gamma^2}{2}}e^{\gamma N_{\rho}}e^{\gamma M}I^{\mathrm{H}}(\frac{\eta}{\gamma}\ln t)>t\}\cap\{\gamma M<\eta\ln t\}\subset\{\rho^{2-\frac{\gamma^2}{2}}e^{\gamma N_{\rho}}I^{\mathrm{H}}(\infty)>t^{1-\eta}\},
    \end{equation*}
    in such a way that for $p>0$,
    \begin{equation*}
    \begin{split}
        \mathbb{E}\left[\frac{\mathbf{1}_{\{\rho^{2-\frac{\gamma^2}{2}}e^{\gamma N_{\rho}}e^{\gamma M}I^{\mathrm{H}}(\frac{\eta}{\gamma}\ln t)>t\}}\mathbf{1}_{\{\gamma M<\eta\ln t\}}}{\rho^{1-\frac{\gamma^2}{4}}e^{\frac{\gamma}{2}N_{\rho}}e^{\frac{\gamma}{2}M}I^{\partial}(\infty)}\right]&\leq \mathbb{E}\left[\frac{\mathbf{1}_{\{\rho^{2-\frac{\gamma^2}{2}}e^{\gamma N_{\rho}}I^{\mathrm{H}}(\infty)>t^{1-\eta}\}}}{\rho^{1-\frac{\gamma^2}{4}}e^{\frac{\gamma}{2}N_{\rho}}e^{\frac{\gamma}{2}M}I^{\partial}(\infty)}\right]\\
        &\leq Ct^{-(1-\eta)p}\mathbb{E}\left[(\rho^{2-\frac{\gamma^2}{2}}e^{\gamma N_{\rho}})^{p-\frac{1}{2}}\right]\mathbb{E}\left[\frac{I^{\mathrm{H}}(\infty)^{p}}{I^{\partial}(\infty)}\right]\\
        &=C\rho^{(p-\frac{1}{2})(2-\gamma^2p)}t^{-(1-\eta)p}\mathbb{E}\left[\frac{I^{\mathrm{H}}(\infty)^{p}}{I^{\partial}(\infty)}\right],
    \end{split}
    \end{equation*}
    where we factorized out the terms with $M$ and $N_\rho$ by independence. With our choice of $p$ in~\eqref{eq:ChoiceOfParameters}, we have $(1-\eta)p>\frac{2}{\gamma^2}+\delta$ and the final expectation term is finite when $p$ is only slightly above $\frac{2}{\gamma^2}$ by Lemma~\ref{lemm:finiteness_of_the_tail_profile_constant}. The $\rho$-coefficient is $\kappa=-(p-\frac{1}{2})(2-\gamma^2p)\in (0,1)$ when $p$ is only slightly above $\frac{2}{\gamma^2}$ as well.
    
    \item For the first expectation term above, we use the exact distribution of $M$ (see Section~\ref{subse:explicit_calculations_on_the_maximum_of_a_drifted_brownian_motion}) to get
    \begin{equation}\label{eq:UseOfExplicitLawOfM}
    \begin{split}
        \mathbb{E}\left[\frac{\mathbf{1}_{\{\rho^{2-\frac{\gamma^2}{2}}e^{\gamma N_{\rho}}e^{\gamma M}I^{\mathrm{H}}(\frac{\eta}{\gamma}\ln t)>t\}}}{\rho^{1-\frac{\gamma^2}{4}}e^{\frac{\gamma}{2}N_{\rho}}e^{\frac{\gamma}{2}M}I^{\partial}(\infty)}\right]=(1-\frac{\gamma^2}{4})\mathbb{E}\left[\frac{I^{\mathrm{H}}(\frac{\eta}{\gamma}\ln t)^{\frac{2}{\gamma^2}}}{I^{\partial}(\infty)}\right]t^{-\frac{2}{\gamma^2}}.
    \end{split}
    \end{equation}
    Notice that the last expression is independent of the choice of $\rho$. It remains to show that the cut-off at $I^{\mathrm{H}}(\frac{\eta}{\gamma}\ln t)$ is a good approximation of $I^{\mathrm{H}}(\infty)$, i.e.
    \begin{equation*}
        \mathbb{E}\left[\frac{I^{\mathrm{H}}(\infty)^{\frac{2}{\gamma^2}}}{I^{\partial}(\infty)}\right](1-o(t^{-\delta}))\leq \mathbb{E}\left[\frac{I^{\mathrm{H}}(\frac{\eta}{\gamma}\ln t)^{\frac{2}{\gamma^2}}}{I^{\partial}(\infty)}\right]\leq \mathbb{E}\left[\frac{I^{\mathrm{H}}(\infty)^{\frac{2}{\gamma^2}}}{I^{\partial}(\infty)}\right].
    \end{equation*}
    The second inequality follows from the monotonicity of $x\mapsto I^{\mathrm{H}}(x)$, see Remark~\ref{rema:TrivialMonotonicity}. The first inequality needs some careful analysis generalizing that in~\cite{Rhodes_2019}. Write with Markov's property the difference
    \begin{equation}\label{eq:DefinitionB_1}
    \begin{split}
        I^{\mathrm{H}}(\infty)-I^{\mathrm{H}}(\frac{\eta}{\gamma}\ln t)&=t^{-\eta}\int_{-\infty}^{0}e^{\gamma\widehat{B}^{\gamma}_s}Z^{\mathrm{H}}_{s-L_{-\frac{\eta}{\gamma}\ln t}}ds,\\
        I^{\partial}(\infty)-I^{\partial}(\frac{\eta}{\gamma}\ln t)&=t^{-\frac{\eta}{2}}\int_{-\infty}^{0}e^{\frac{\gamma}{2}\widehat{B}^{\gamma}_s}Z^{\partial}_{s-L_{-\frac{\eta}{\gamma}\ln t}}ds,
    \end{split}
    \end{equation}
    where the process $\widehat{B}^{\gamma}_s=B^{\gamma}_{s-L_{-\frac{\eta}{\gamma}\ln t}}+\frac{\eta}{\gamma}\ln t$ for $s\leq 0$ is independent of $(B^{\gamma}_s,Z^{\mathrm{H}}_s,Z^{\partial}_s)_{s\geq 0}$ and distributed as $(B^{\gamma}_s)_{s\leq 0}$. Set
    \begin{equation*}
        B^{\mathrm{H}}=\int_{-\infty}^{0}e^{\gamma\widehat{B}^{\gamma}_s}Z_{s-L_{-\frac{\eta}{\gamma}\ln t}}ds\quad\text{and}\quad B^{\partial}=\int_{-\infty}^{0}e^{\frac{\gamma}{2}\widehat{B}^{\gamma}_s}Z^{\partial}_{s-L_{-\frac{\eta}{\gamma}\ln t}}ds.
    \end{equation*}
    
    We distinguish two cases:
    \begin{enumerate}
        \item If $\frac{2}{\gamma^2}>1$, Minkowski's inequality for the $L_{\frac{2}{\gamma^2}}$-norm yields
        \begin{equation*}
        \begin{split}
            &\mathbb{E}\left[\frac{I^{\mathrm{H}}(\infty)^{\frac{2}{\gamma^2}}}{I^{\partial}(\infty)}\right]-\mathbb{E}\left[\frac{I^{\mathrm{H}}(\frac{\eta}{\gamma}\ln t)^{\frac{2}{\gamma^2}}}{I^{\partial}(\infty)}\right]\\
            \leq{}&\mathbb{E}\left[\left(\frac{I^{\mathrm{H}}(\infty)}{I^{\partial}(\infty)^{\frac{\gamma^2}{2}}}\right)^{\frac{2}{\gamma^2}}\right]-\mathbb{E}\left[\left(\frac{I^{\mathrm{H}}(\frac{\eta}{\gamma}\ln t)}{I^{\partial}(\infty)^{\frac{\gamma^2}{2}}}\right)^{\frac{2}{\gamma^2}}\right]\\
            ={}&\mathbb{E}\left[\left(\frac{I^{\mathrm{H}}(\frac{\eta}{\gamma}\ln t)+t^{-\eta}B^{\mathrm{H}}}{I^{\partial}(\infty)^{\frac{\gamma^2}{2}}}\right)^{\frac{2}{\gamma^2}}\right]-\mathbb{E}\left[\left(\frac{I^{\mathrm{H}}(\frac{\eta}{\gamma}\ln t)}{I^{\partial}(\infty)^{\frac{\gamma^2}{2}}}\right)^{\frac{2}{\gamma^2}}\right]\\
            \leq{}&\left(\mathbb{E}\left[\left(\frac{I^{\mathrm{H}}(\frac{\eta}{\gamma}\ln t)}{I^{\partial}(\infty)^{\frac{\gamma^2}{2}}}\right)^{\frac{2}{\gamma^2}}\right]^{\frac{\gamma^2}{2}}+\mathbb{E}\left[\left(\frac{t^{-\eta}B^{\mathrm{H}}}{I^{\partial}(\infty)^{\frac{\gamma^2}{2}}}\right)^{\frac{2}{\gamma^2}}\right]^{\frac{\gamma^2}{2}}\right)^{\frac{2}{\gamma^2}}-\mathbb{E}\left[\left(\frac{I^{\mathrm{H}}(\frac{\eta}{\gamma}\ln t)}{I^{\partial}(\infty)^{\frac{\gamma^2}{2}}}\right)^{\frac{2}{\gamma^2}}\right].
        \end{split}
        \end{equation*}
        Notice that the last expression is of the form $(\alpha_t+\beta_t)^{\frac{2}{\gamma^2}}-\alpha_t^{\frac{2}{\gamma^2}}$ with $\frac{\beta_t}{\alpha_t}\leq c$ uniformly for sufficiently large $t\geq t_0$. Indeed, notice that $\mathbb{E}\left[\frac{(B^{\mathrm{H}})^{\frac{2}{\gamma^2}}}{B^{\partial}}\right]<\infty$ using Corollary~\ref{coro:B_B} and $I^{\partial}(\infty)\geq t^{-\frac{\eta}{2}}B^{\partial}$, so as $t\to\infty$, $\alpha_t$ converges to a positive constant while $\beta_t$ is at most $O(t^{-\eta(1-\frac{\gamma^2}{4})})$. Since $(1+x)^{\frac{2}{\gamma^2}}-1\leq Cx$ for some $C$ uniformly in $0\leq x\leq c$, we have $(\alpha_t+\beta_t)^{\frac{2}{\gamma^2}}-\alpha_t^{\frac{2}{\gamma^2}}\leq C\alpha_t^{\frac{2}{\gamma^2}-1}\beta_t$. Plugging this in the above display yields
        \begin{equation*}
        \begin{split}
            \mathbb{E}\left[\frac{I^{\mathrm{H}}(\infty)^{\frac{2}{\gamma^2}}}{I^{\partial}(\infty)}\right]-\mathbb{E}\left[\frac{I^{\mathrm{H}}(\frac{\eta}{\gamma}\ln t)^{\frac{2}{\gamma^2}}}{I^{\partial}(\infty)}\right]&\leq C\mathbb{E}\left[\frac{I^{\mathrm{H}}(\frac{\eta}{\gamma}\ln t)^{\frac{2}{\gamma^2}}}{I^{\partial}(\infty)}\right]^{1-\frac{\gamma^2}{2}}\mathbb{E}\left[\frac{(t^{-\eta}B^{\mathrm{H}})^{\frac{2}{\gamma^2}}}{I^{\partial}(\infty)}\right]^{\frac{\gamma^2}{2}}\\
            &\leq Ct^{-\eta}\mathbb{E}\left[\frac{I^{\mathrm{H}}(\infty)^{\frac{2}{\gamma^2}}}{I^{\partial}(\infty)}\right]^{1-\frac{\gamma^2}{2}}\mathbb{E}\left[\frac{(B^{\mathrm{H}})^{\frac{2}{\gamma^2}}}{I^{\partial}(\infty)}\right]^{\frac{\gamma^2}{2}}.
        \end{split}
        \end{equation*}
        The first expectation is finite by Lemma~\ref{lemm:finiteness_of_the_tail_profile_constant}. For the second expectation term, we again need to invoke the main theorem of~\cite{Huang:2025aa}. First write
        \begin{equation*}
            I^{\partial}(\infty)=I^{\partial}(\frac{\eta}{\gamma}\ln t)+t^{-\frac{\eta}{2}}B^{\partial}
        \end{equation*}
        so that Hölder's inequality yields, for any $\eta'\in(0,1)$,
        \begin{equation*}
            \mathbb{E}\left[\frac{(B^{\mathrm{H}})^{\frac{2}{\gamma^2}}}{I^{\partial}(\infty)}\right]=\mathbb{E}\left[\frac{(B^{\mathrm{H}})^{\frac{2}{\gamma^2}}}{I^{\partial}(\frac{\eta}{\gamma}\ln t)+t^{-\frac{\eta}{2}}B^{\partial}}\right]\leq\mathbb{E}\left[\frac{(B^{\mathrm{H}})^{\frac{2}{\gamma^2}}}{I^{\partial}(\frac{\eta}{\gamma}\ln t)^{1-\eta'}(t^{-\frac{\eta}{2}}B^{\partial})^{\eta'}}\right]=t^{\eta'\frac{\eta}{2}}\mathbb{E}\left[\frac{(B^{\mathrm{H}})^{\frac{2}{\gamma^2}}}{I^{\partial}(\frac{\eta}{\gamma}\ln t)^{1-\eta'}(B^{\partial})^{\eta'}}\right].
        \end{equation*}
        Since $I^{\partial}(\frac{\eta}{\gamma}\ln t)$ has negative moments of any order and $\mathbb{E}\left[\frac{(B^{\mathrm{H}})^{\kappa\frac{2}{\gamma^2}}}{(B^{\partial})^{\kappa\eta'}}\right]<\infty$ for $\kappa$ close to $1^+$ in view of Corollary~\ref{coro:B_B} (the condition here is $\kappa(\frac{2}{\gamma^2}-\frac{\eta'}{2})<\frac{2}{\gamma^2}$), another standard use of Hölder's inequality proves our claim provided that $\eta'$ is chosen small enough.
        \item If $0<\frac{2}{\gamma^2}\leq 1$, we use the subadditivity inequality instead of Minkowski's inequality in the above equations. We omit the details as this modification has been used multiple times in this series of papers.
    \end{enumerate}
\end{enumerate}

The proof of the lower bound is now complete: the upshot is
\begin{equation}\label{eq:Upshot_Lowerbound}
    \mathbb{E}\left[\frac{\mathbf{1}_{\{\mu^{\mathrm{H}}_0(Q(0,\rho))>t\}}}{\mu^{\partial}_0(I(0,\rho))}\right]\geq (1-\frac{\gamma^2}{4})\mathbb{E}\left[\frac{I^{\mathrm{H}}(\frac{\eta}{\gamma}\ln t)^{\frac{2}{\gamma^2}}}{I^{\partial}(\infty)}\right]t^{-\frac{2}{\gamma^2}}+o(t^{-\frac{2}{\gamma^2}-\delta}).
\end{equation}

\subsection{Upper bound estimate}\label{subse:upper_bound_estimate}
We now establish the corresponding upper bound for~\eqref{eq:LocalizedExpression_I}. The strategy is similar: we want get rid of the $M$ dependence in either $I^{\mathrm{H}}(M)$ or $I^{\partial}(M)$, in order to use the exact law of $M$ recalled in Section~\ref{subse:explicit_calculations_on_the_maximum_of_a_drifted_brownian_motion} to get the desired formulas.

Indeed, if we can establish that
\begin{equation}\label{eq:UpperBoundApproximation}
    \mathbb{E}\left[\frac{\mathbf{1}_{\{\rho^{2-\frac{\gamma^2}{2}}e^{\gamma N_{\rho}}e^{\gamma M}I^{\mathrm{H}}(M)>t\}}}{\rho^{1-\frac{\gamma^2}{4}}e^{\frac{\gamma}{2}N_{\rho}}e^{\frac{\gamma}{2}M}I^{\partial}(M)}\right]-\mathbb{E}\left[\frac{\mathbf{1}_{\{\rho^{2-\frac{\gamma^2}{2}}e^{\gamma N_{\rho}}e^{\gamma M}I^{\mathrm{H}}(M)>t\}}}{\rho^{1-\frac{\gamma^2}{4}}e^{\frac{\gamma}{2}N_{\rho}}e^{\frac{\gamma}{2}M}I^{\partial}(\infty)}\right]=o(t^{-\frac{2}{\gamma^2}-\delta})
\end{equation}
where we replaced in the denominator $I^{\partial}(M)$ by $I^{\partial}(\infty)$, the upper bound follows by observing then that we can further replace in the nominator $I^{\mathrm{H}}(M)$ by $I(\infty)$ to upper bound
\begin{equation*}
    \mathbb{E}\left[\frac{\mathbf{1}_{\{\rho^{2-\frac{\gamma^2}{2}}e^{\gamma N_{\rho}}e^{\gamma M}I^{\mathrm{H}}(M)>t\}}}{\rho^{1-\frac{\gamma^2}{4}}e^{\frac{\gamma}{2}N_{\rho}}e^{\frac{\gamma}{2}M}I^{\partial}(\infty)}\right]\leq \mathbb{E}\left[\frac{\mathbf{1}_{\{\rho^{2-\frac{\gamma^2}{2}}e^{\gamma N_{\rho}}e^{\gamma M}I^{\mathrm{H}}(\infty)>t\}}}{\rho^{1-\frac{\gamma^2}{4}}e^{\frac{\gamma}{2}N_{\rho}}e^{\frac{\gamma}{2}M}I^{\partial}(\infty)}\right]=(1-\frac{\gamma^2}{4})\mathbb{E}\left[\frac{I^{\mathrm{H}}(\infty)^{\frac{2}{\gamma^2}}}{I^{\partial}(\infty)}\right]t^{-\frac{2}{\gamma^2}},
\end{equation*}
where the last equality is obtained using the explicit law of the maximum $M$ of Section~\ref{subse:explicit_calculations_on_the_maximum_of_a_drifted_brownian_motion}, similar to~\eqref{eq:UseOfExplicitLawOfM}. Notice that the last term in the above display is the desired tail coefficient of our main Theorem~\ref{th:main_result}.

The rest of this section is devoted to proving~\eqref{eq:UpperBoundApproximation}. We again follow the strategy of~\cite[Section~3.6]{Rhodes_2019} and start by defining the ``cost of replacement'' in~\eqref{eq:UpperBoundApproximation},
\begin{equation*}
    C(t)=\mathbb{E}\left[\frac{\mathbf{1}_{\{\rho^{2-\frac{\gamma^2}{2}}e^{\gamma N_{\rho}}e^{\gamma M}I^{\mathrm{H}}(M)>t\}}}{\rho^{1-\frac{\gamma^2}{4}}e^{\frac{\gamma}{2}N_{\rho}}e^{\frac{\gamma}{2}M}}\left(\frac{1}{I^{\partial}(M)}-\frac{1}{I^{\partial}(\infty)}\right)\right].
\end{equation*}
Similarly to~\eqref{eq:DefinitionB_1} but this time applied at time $L_{-M}$ instead of $L_{-\frac{\eta}{\gamma}\ln t}$, we define
\begin{equation*}
    \widehat{B}^{\gamma}_s=B^{\gamma}_{s-L_{-M}}+M
\end{equation*}
so that
\begin{equation}\label{eq:DefinitionB_2}
\begin{split}
    I^{\mathrm{H}}(\infty)-I^{\mathrm{H}}(M)&=e^{-\gamma M}\int_{-\infty}^{0}e^{\gamma\widehat{B}^{\gamma}_s}Z^{\mathrm{H}}_{s-L_{-M}}ds=e^{-\gamma M}B^{\mathrm{H}},\\
    I^{\partial}(\infty)-I^{\partial}(M)&=e^{-\frac{\gamma M}{2}}\int_{-\infty}^{0}e^{\frac{\gamma}{2}\widehat{B}^{\gamma}_s}Z^{\partial}_{s-L_{-M}}ds=e^{-\frac{\gamma M}{2}}B^{\partial}.
\end{split}
\end{equation}
We stress that $\widehat{B}, B^{\mathrm{H}}$ and $B^{\partial}$ are not the same as the previously defined ones in the proof of the lower bound, but they share similar independence properties.

To fix ideas, we choose the parameters below such that with some small $\delta>0$ and $\eta\in(0,1)$,
\begin{equation}\label{eq:Conditions_Parameters_C12}
    \frac{2}{\gamma^2}<p<\min(\frac{2}{\gamma^2}+\Delta p, \frac{2}{\gamma^2}+\frac{1}{2}, \frac{4}{\gamma^2}),\quad p(1-\eta)>\frac{2}{\gamma^2}(1-\eta)+\delta\quad\text{and}\quad \eta(\frac{2}{\gamma^2}+\frac{1}{2})>\frac{2}{\gamma^2}+\delta.
\end{equation}
One checks by elementary calculation that this is indeed possible with $\gamma\in(0,2)$ for any small enough $\Delta p>0$: it suffice to first choose $p$ slightly above $\frac{2}{\gamma^2}$, then $\eta$ slightly below $1$ and lastly fix $\delta>0$ accordingly.

We proceed by separating two cases depending on the relation between $M$ and $\frac{\eta}{\gamma}\ln t$.
\begin{enumerate}
    \item Suppose that $M>\frac{\eta}{\gamma}\ln t$ and consider
    \begin{equation*}
        C_1(t)=\mathbb{E}\left[\frac{\mathbf{1}_{\{\rho^{2-\frac{\gamma^2}{2}}e^{\gamma N_{\rho}}e^{\gamma M}I^{\mathrm{H}}(M)>t\}}}{\rho^{1-\frac{\gamma^2}{4}}e^{\frac{\gamma}{2}N_{\rho}}e^{\frac{\gamma}{2}M}}\left(\frac{1}{I^{\partial}(M)}-\frac{1}{I^{\partial}(\infty)}\right)\mathbf{1}_{\{M>\frac{\eta}{\gamma}\ln t\}}\right].
    \end{equation*}
    Adapting the proof strategy of~\cite[Equation~(3.15)]{Rhodes_2019}, we introduce the auxiliary function
    \begin{equation*}
        F(u)=\mathbb{E}\left[\frac{1}{\rho^{1-\frac{\gamma^2}{4}}e^{\frac{\gamma}{2}N_{\rho}}}\mathbf{1}_{\{\rho^{2-\frac{\gamma^2}{2}}e^{\gamma N_{\rho}}>u\}}\right]
    \end{equation*}
    so that
    \begin{equation*}
        C_1(t)\leq\mathbb{E}\left[F\left(\frac{t}{e^{\gamma M}I^{\mathrm{H}}(M)}\right)\frac{B^{\partial}}{(I^{\partial}(M))^2}\frac{1}{e^{\gamma M}}\mathbf{1}_{\{M>\frac{\eta}{\gamma}\ln t\}}\right].
    \end{equation*}
    We continue by separating again two cases according to the value of the argument $\frac{t}{e^{\gamma M}I^{\mathrm{H}}(M)}$ compared to some power of $\rho$. For this, choose a parameter $0<a<2+\frac{\gamma^2}{2}$.
    \begin{enumerate}
        \item Suppose furthermore that $\frac{t}{e^{\gamma M}I^{\mathrm{H}}(M)}\leq \rho^{a}$. In this case, simply discard the indicator function in the definition of $F(u)$ produces the upper bound
        \begin{equation*}
        \begin{split}
            C_{1'}(t)&=\mathbb{E}\left[F\left(\frac{t}{e^{\gamma M}I^{\mathrm{H}}(M)}\right)\frac{B^{\partial}}{(I^{\partial}(M))^2}\frac{1}{e^{\gamma M}}\mathbf{1}_{\{M>\frac{\eta}{\gamma}\ln t\}}\mathbf{1}_{\{\frac{t}{e^{\gamma M}I^{\mathrm{H}}(M)}\leq \rho^{a}\}}\right]\\
            &\leq\mathbb{E}\left[\frac{1}{\rho^{1-\frac{\gamma^2}{4}}e^{\frac{\gamma}{2}N_{\rho}}}\frac{B^{\partial}}{(I^{\partial}(M))^2}\frac{1}{e^{\gamma M}}\mathbf{1}_{\{M>\frac{\eta}{\gamma}\ln t\}}\mathbf{1}_{\{\frac{t}{e^{\gamma M}I^{\mathrm{H}}(M)}\leq \rho^{a}\}}\right]\\
            &\leq \rho^{-1+qa}t^{-q}\mathbb{E}\left[\frac{B^{\partial}(I^{\mathrm{H}}(M))^{q}}{(I^{\partial}(M))^2}\frac{1}{e^{(1-q)\gamma M}}\mathbf{1}_{\{M>\frac{\eta}{\gamma}\ln t\}}\right].
        \end{split}
        \end{equation*}
        for any $0<q<1$. Notice that conditioning on $M$, following the same argument below~\cite[Equation~(3.17)]{Rhodes_2019} shows that
        \begin{equation*}
            \mathbb{E}\left[\frac{B^{\partial}(I^{\mathrm{H}}(M))^{q}}{(I^{\partial}(M))^2}\mid M\right]
        \end{equation*}
        is bounded by a constant. Indeed, using Hölder's inequality with conjugate parameters $m,m'$ where $m$ is slightly below $\frac{2}{\gamma^2}$,
        \begin{equation*}
            \mathbb{E}\left[\frac{B^{\partial}(I^{\mathrm{H}}(M))^{q}}{(I^{\partial}(M))^2}\mid M\right]\leq \mathbb{E}\left[(B^{\partial})^{m}\right]^{1/m}\mathbb{E}\left[\frac{(I^{\mathrm{H}}(M))^{qm'}}{(I^{\partial}(M))^{2m'}}\mid M\right]^{1/m'}<\infty
        \end{equation*}
        where we used the independence between $B^{\partial}$ and $M$, and the fact that $B^{\partial}$ is a variant of classical Gaussian multiplicative chaos therefore has some positive moments (the moment bound is $\frac{2}{\gamma^2}$ due to the extra factor $2$ in the log-covariance), the last expectation is finite as $2(qm')<2m'$ thanks to Corollary~\ref{coro:supM} in Section~\ref{subse:some_preparations}. Plugging this in the upper bound of $C_{1'}(t)$ above yields
        \begin{equation*}
            C_{1'}(t)\leq C\rho^{-1+qa}t^{-q}\mathbb{E}\left[\frac{1}{e^{(1-q)\gamma M}}\mathbf{1}_{\{M>\frac{\eta}{\gamma}\ln t\}}\right]=C\rho^{-1+qa}t^{-(1-\eta)q}t^{-\frac{Q\eta}{\gamma}},
        \end{equation*}
        where the last equality follows from the explicit law of $M$ recalled in Section~\ref{subse:explicit_calculations_on_the_maximum_of_a_drifted_brownian_motion}. Under conditions~\eqref{eq:Conditions_Parameters_C12}, one verifies that this is indeed $\rho^{-\kappa}o(t^{-\frac{2}{\gamma^2}-\delta})$ with $\kappa<1$ and $\delta>0$.
        \item Suppose on the contrary that $\frac{t}{e^{\gamma M}I^{\mathrm{H}}(M)}>\rho^{a}$. This time, we use Girsanov's transformation to bound $F(u)$ with $u=\frac{t}{e^{\gamma M}I^{\mathrm{H}}(M)}>\rho^{a}$ as follows:
        \begin{equation*}
            F(u)=\rho^{-1}\mathbb{P}\left[\rho^{2+\frac{\gamma^2}{2}}e^{\gamma N_\rho}>u\right]\leq \rho^{-1}\mathbb{P}[e^{\gamma N_{\rho}}>\rho^{-2-\frac{\gamma^2}{2}}\rho^{a}]\leq \rho^{-1+\left(\frac{2+\frac{\gamma^2}{2}-a}{2\gamma}\right)^2}
        \end{equation*}
        if $a<2+\frac{\gamma^2}{2}$. Therefore, we get the following upper bound with $\kappa<1$,
        \begin{equation*}
        \begin{split}
            C_{1''}(t)&=\mathbb{E}\left[F\left(\frac{t}{e^{\gamma M}I^{\mathrm{H}}(M)}\right)\frac{B^{\partial}}{(I^{\partial}(M))^2}\frac{1}{e^{\gamma M}}\mathbf{1}_{\{M>\frac{\eta}{\gamma}\ln t\}}\mathbf{1}_{\{\frac{t}{e^{\gamma M}I^{\mathrm{H}}(M)}>\rho^{a}\}}\right]\\
            &\leq C{\rho}^{-\kappa}\mathbb{E}\left[\frac{B^{\partial}}{(I^{\partial}(M))^2}\frac{1}{e^{\gamma M}}\mathbf{1}_{\{M>\frac{\eta}{\gamma}\ln t\}}\right]\\
            &=C'\rho^{-\kappa}t^{-\frac{Q\eta}{\gamma}}
        \end{split}
        \end{equation*}
        with $\frac{Q\eta}{\gamma}>\frac{2}{\gamma^2}+\delta$ under conditions~\eqref{eq:Conditions_Parameters_C12}, where the last equality uses the explicit law of $M$ recalled in Section~\ref{subse:explicit_calculations_on_the_maximum_of_a_drifted_brownian_motion} and the fact that
        \begin{equation*}
            \mathbb{E}\left[\frac{B^{\partial}}{(I^{\partial}(M))^2}\mid M\right]<\infty
        \end{equation*}
        as in the previous paragraph, using Hölder's inequality and that $I^{\partial}(M)$ has all negative moments finite.
    \end{enumerate}
    The proof in the case $M>\frac{\eta}{\gamma}\ln t$ is finished using that $C_1(t)=C_{1'}(t)+C_{1''}(t)$, as the exponents in the powers of $\rho$ and $t$ are as desired.
    
    \item Suppose now that $M\leq\frac{\eta}{\gamma}\ln t$ and consider
    \begin{equation*}
        C_2(t)=\mathbb{E}\left[\frac{\mathbf{1}_{\{\rho^{2-\frac{\gamma^2}{2}}e^{\gamma N_{\rho}}e^{\gamma M}I^{\mathrm{H}}(M)>t\}}}{\rho^{1-\frac{\gamma^2}{4}}e^{\frac{\gamma}{2}N_{\rho}}e^{\frac{\gamma}{2}M}}\left(\frac{1}{I^{\partial}(M)}-\frac{1}{I^{\partial}(\infty)}\right)\mathbf{1}_{\{M\leq \frac{\eta}{\gamma}\ln t\}}\right].
    \end{equation*}
    It suffices to bound $\frac{1}{I^{\partial}(M)}-\frac{1}{I^{\partial}(\infty)}$ by $\frac{1}{I^{\partial}(M)}$ and
    \begin{equation*}
    \begin{split}
        C_2(t)&\leq\mathbb{E}\left[\frac{\mathbf{1}_{\{\rho^{2-\frac{\gamma^2}{2}}e^{\gamma N_{\rho}}e^{\gamma M}I^{\mathrm{H}}(M)>t\}}}{\rho^{1-\frac{\gamma^2}{4}}e^{\frac{\gamma}{2}N_{\rho}}e^{\frac{\gamma}{2}M}}\frac{1}{I^{\partial}(M)}\mathbf{1}_{\{M\leq \frac{\eta}{\gamma}\ln t\}}\right]\\
        &\leq t^{-p}\mathbb{E}\left[\frac{(\rho^{2-\frac{\gamma^2}{2}}e^{\gamma N_{\rho}}e^{\gamma M}I^{\mathrm{H}}(M))^{p}}{\rho^{1-\frac{\gamma^2}{4}}e^{\frac{\gamma}{2}N_{\rho}}e^{\frac{\gamma}{2}M}}\frac{1}{I^{\partial}(M)}\mathbf{1}_{\{M\leq \frac{\eta}{\gamma}\ln t\}}\right]\\
        &=t^{-p}\mathbb{E}\left[(\rho^{2-\frac{\gamma^2}{2}}e^{\gamma N_{\rho}})^{p-\frac{1}{2}}\right]\mathbb{E}\left[e^{(p-\frac{1}{2})\gamma M}\frac{I^{\mathrm{H}}(M)^{p}}{I^{\partial}(M)}\mathbf{1}_{\{M\leq \frac{\eta}{\gamma}\ln t\}}\right]\\
        &\leq C\rho^{(p-\frac{1}{2})(2-\gamma^2p)} t^{-p}\sup_{x>0}\mathbb{E}\left[\frac{I^{\mathrm{H}}(x)^{p}}{I^{\partial}(x)}\right]\mathbb{E}\left[e^{(p-\frac{1}{2})\gamma M}\mathbf{1}_{\{M\leq \frac{\eta}{\gamma}\ln t\}}\right]\\
        &=C'\rho^{(p-\frac{1}{2})(2-\gamma^2p)} t^{-p}\mathbb{E}\left[e^{(p-\frac{1}{2})\gamma M}\mathbf{1}_{\{M\leq \frac{\eta}{\gamma}\ln t\}}\right]\\
        &=C''\rho^{(p-\frac{1}{2})(2-\gamma^2p)} t^{-p}t^{\frac{\eta}{\gamma}(\gamma p-\frac{2}{\gamma})},
    \end{split}
    \end{equation*}
    where in the penultimate equality we used Corollary~\ref{coro:supM}, and in the last equality we used the explicit distribution of $M$. Under conditions~\eqref{eq:ChoiceOfParameters} and~\eqref{eq:Conditions_Parameters_C12}, the bound is $C_\rho o(t^{-\frac{2}{\gamma^2}-\delta})$ so that the $t$-coefficient is as desired. The $\rho$-coefficient is $\kappa=-(p-\frac{1}{2})(2-\gamma^2p)\in (0,1)$ when $p$ is only slightly above $\frac{2}{\gamma^2}$ as well.
\end{enumerate}

The proof of the upper bound is complete since the left hand side of~\eqref{eq:UpperBoundApproximation} is $C(t)=C_1(t)+C_2(t)$:
\begin{equation}\label{eq:Upshot_Upperbound}
    \mathbb{E}\left[\frac{\mathbf{1}_{\{\mu^{\mathrm{H}}_0(Q(0,\rho))>t\}}}{\mu^{\partial}_0(I(0,\rho))}\right]\leq (1-\frac{\gamma^2}{4})\mathbb{E}\left[\frac{I^{\mathrm{H}}(\frac{\eta}{\gamma}\ln t)^{\frac{2}{\gamma^2}}}{I^{\partial}(\infty)}\right]t^{-\frac{2}{\gamma^2}}+o(t^{-\frac{2}{\gamma^2}-\delta}).
\end{equation}
Combining~\eqref{eq:Upshot_Lowerbound} and~\eqref{eq:Upshot_Upperbound} finishes our proof of Theorem~\ref{th:main_result}.

\section{Extension to non exact-scaling kernels}\label{sec:extension_to_non_exact_scaling_kernels}
Lastly, we record the expression of the tail profile coefficient of the bulk Gaussian multiplicative chaos measure when the underlying covariance kernel of the log-correlated Gaussian field is no longer assumed exact scale-invariant.
\begin{coro}
Let $Q_r$ the Carleson cube defined in the main text, and $\mu^{H,g}$ the bulk Gaussian multiplicative chaos defined in~\eqref{eq:DefinitionBulkMeasure} with a general Neumann-type covariance kernel
\begin{equation}\label{eq:GeneralNeumannKernel}
    K(z,w)=-\ln|z-w||z-\overline{w}|+g(z,w)
\end{equation}
where $g$ is assumed locally Hölder and positive definite on $Q_r\times Q_r$. Then as $t$ goes to $\infty$,
\begin{equation*}
    \mathbb{P}[\mu^{\mathrm{H},g}(Q_r)>t]\sim C_{(1,0)}\left(\int_{-r}^{r}e^{(\frac{2}{\gamma^2}-1)g(v,v)}dv\right)t^{-\frac{2}{\gamma^2}}.
\end{equation*}
where $C_{(1,0)}$ is the constant defined in Theorem~\ref{th:main_result}.
\end{coro}
The proof (omitted here) is a direct adaptation of~\cite[Corollary~2.4]{Rhodes_2019} using the techniques and estimates presented in the proof of Theorem~\ref{th:main_result}. It is likely that one can relax the assumption that $g$ is positive definite (which simplifies the proof since we can add an independent Gaussian field to $X_{\mathrm{C}}$) using techniques from~\cite{wong2020universal}: this approach will be investigated in a future work.

\begin{rema}\label{rema:Final_Remark}
Notice that the extra term $\left(\int_{-r}^{r}e^{(\frac{2}{\gamma^2}-1)g(v,v)}dv\right)$ in the exact formula of the above corollary only depends on the (diagonal) covariance structure on the boundary $v\in\mathbb{R}$, which is as expected as the main contribution to the right tail is concentrated around the boundary as discussed in Remark~\ref{rema:only_boundary_covariance_matters}.
\end{rema}

\appendix
\section{On some variants of bulk/boundary quotients of Gaussian multiplicative chaos measures}
In this appendix, we gather some variants and slightly finer estimates of the main result in~\cite{Huang:2025aa}.

We mainly follows~\cite[Section~5.6.3]{Huang:2025aa}, and uses notations are preliminary results in this companion paper. Suppose that $Q_r\subset\overline{\mathbb{H}}$ is the Carleson cube considered in the main text, $I_r=[-r,r]\subset\mathbb{R}$ its boundary with singularity, and $v\in[-r,r]$ a point where we perform the localization trick at the boundary. Without loss of generality, suppose that $v\in[0,r]$ and denote by $2\rho=r-v$ the distance of $v$ to the nearest boundary end point. Recall also that $Q(v,\rho)$ is the half disk of center $v$ and radius $\rho$ contained in $\mathbb{H}$, and $I(v,\rho)=[v-\rho,v+\rho]$. By~\cite[Theorem~2]{Huang:2025aa} or more precisely~\cite[Proposition~22]{Huang:2025aa}, for $p,q\geq 0$ and $p<\min(\frac{2}{\gamma^2}+\frac{q}{2},\frac{4}{\gamma^2})$,
\begin{equation*}
    \mathbb{E}\left[\frac{\mu^{\mathrm{H}}_\rho(Q(v,\rho)^{c})^{p}}{\mu^{\partial}_\rho(I(v,\rho)^{c})^q}\right]<\infty
\end{equation*}
where $Q(v,\rho)^{c}=Q_r\setminus Q(v,\rho)$ and $I(v,\rho)^{c}=I_r\setminus I(v,\rho)$, and $\mu^{\mathrm{H}}_\rho, \mu^{\partial}_\rho$ are respectively the localized bulk and boundary Gaussian multiplicative chaos measures.

We now study the dependence on $\rho$ of the above expression.
\begin{lemm}
Suppose that $p,q\geq 0$ and $p<\min(\frac{2}{\gamma^2}+\frac{q}{2},\frac{4}{\gamma^2})$. Then the expectation of the localized quotient has the following growth with $\rho$,
\begin{equation*}
    \mathbb{E}\left[\frac{\mu_v^{\mathrm{H}}(Q(v,\rho)^{c})^{p}}{\mu^{\partial}_v(I(v,\rho)^{c})^q}\right]=O(\rho^{\widetilde{\zeta}(p;q)}\vee 1),
\end{equation*}
where $\widetilde{\zeta}(p;q)=\left(2-\frac{\gamma^2}{2}\right)(p-\frac{q}{2})-\gamma^2(p-\frac{q}{2})^2$.
\end{lemm}

\begin{figure}[h]
\centering
\includegraphics[height=15em]{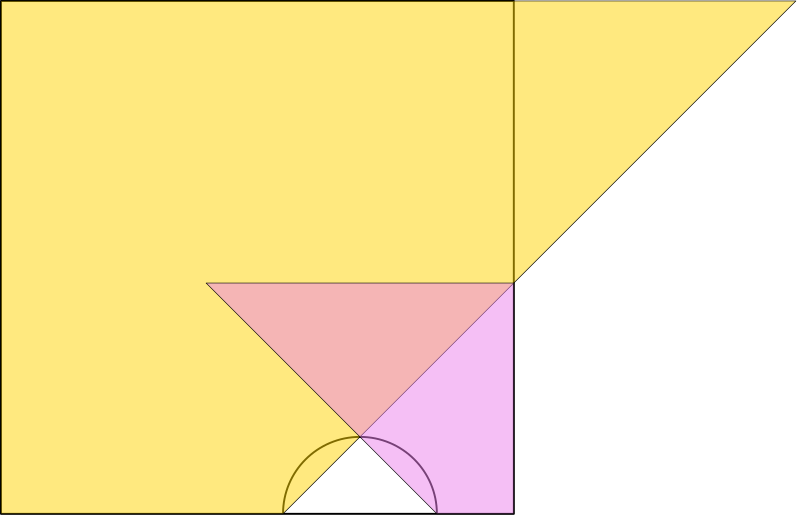}
\caption{The semi-circle $S(v,\rho)$ is centered at $v\in\mathbb{R}$ with radius $\rho=\frac{r-v}{2}$ when $v>0$. The two colored polygones pass through the top of the semi-circle of coordinate $(v,\rho)$ and cover $Q(v,\rho)^{c}=Q_r\setminus Q(v,\rho)$.}
\label{fig:FinalLemma}
\end{figure}

\begin{rema}\label{rema:FinalLemma}
In particular, for $(p,q)=(p,1)$ or $(p,q)=(p,2)$, we have $\widetilde{\zeta}(p,q)>-1$ under the condition for $p$ close to $\min(\frac{2}{\gamma^2}+\frac{q}{2},\frac{4}{\gamma^2})$: these two special cases are using in the proofs in Section~\ref{sec:reduction_to_local_estimates}. Indeed, when $q=1$, we have $\widetilde{\zeta}(p,q)>-1$ for $p\in(0,\frac{2}{\gamma^2}+\frac{1}{2})$; when $q=2$, we have $\widetilde{\zeta}(p,q)>-1$ for $p\in(\frac{1}{2},\frac{2}{\gamma^2}+1)$ and in particular $p=\frac{2}{\gamma^2}$ since $\gamma\in(0,2)$.
\end{rema}

\begin{proof}
We start by a covering of the complement $Q(v,\rho)^{c}$ by two regions as in Figure~\ref{fig:FinalLemma}, where one edge is having a $\frac{\pi}{4}$ angle with the boundary. In particular, the yellow region denoted $Q^{\mathrm{L}}_\rho$ and purple region denoted $Q^{\mathrm{R}}_\rho$ are similar, with the shortest edge of the purple region having size $\rho$. We also denote by $I^{\textrm{L}}_\rho=[-r,v]$ and $I^{\mathrm{R}}_\rho=[v,r]$. Then up to some constant depending only on $p$,
\begin{equation*}
\begin{split}
    \mathbb{E}\left[\frac{\mu^{\mathrm{H}}_v(Q(v,\rho)^{c})^{p}}{\mu^{\partial}_v(I(v,\rho)^{c})^q}\right]&\leq C\left(\mathbb{E}\left[\frac{\mu^{\mathrm{H}}_v(Q^{\mathrm{L}}_\rho)^{p}}{\mu^{\partial}_v(I(v,\rho)^{c})^q}\right]+\mathbb{E}\left[\frac{\mu^{\mathrm{H}}_v(Q^{\mathrm{R}}_\rho)^{p}}{\mu^{\partial}_v(I(v,\rho)^{c})^q}\right]\right)\\
    &\leq C\left(\mathbb{E}\left[\frac{\mu^{\mathrm{H}}_v(Q^{\mathrm{L}}_\rho)^{p}}{\mu^{\partial}_v(I^{\mathrm{L}}_\rho)^q}\right]+\mathbb{E}\left[\frac{\mu^{\mathrm{H}}_v(Q^{\mathrm{R}}_\rho)^{p}}{\mu^{\partial}_v(I^{\mathrm{R}}_\rho)^q}\right]\right).
\end{split}
\end{equation*}

We first consider the last term $\mathbb{E}\left[\frac{\mu_v^{\mathrm{H}}(Q^{\mathrm{R}}_\rho)^{p}}{\mu^{\partial}_v(I^{\mathrm{R}}_\rho)^q}\right]$. By~\cite[Proposition~22]{Huang:2025aa}, it is finite under the condition that $p,q\geq 0$ and $p<\min(\frac{2}{\gamma^2}+\frac{q}{2},\frac{4}{\gamma^2})$. Furthermore, by the exact scaling relation of the localized bulk/boundary quotient~\cite[Lemma~13]{Huang:2025aa}, its dependence on $\rho$ is exactly $C\rho^{\widetilde{\zeta}(p;q)}$ with
\begin{equation*}
    \widetilde{\zeta}(p;q)=\left(2-\frac{\gamma^2}{2}\right)(p-\frac{q}{2})-\gamma^2(p-\frac{q}{2})^2.
\end{equation*}
The first term $\mathbb{E}\left[\frac{\mu^{\mathrm{H}}_v(Q^{\mathrm{L}}_\rho)^{p}}{\mu^{\partial}_v(I^{\mathrm{L}}_\rho)^q}\right]$ can be treated similarly, but since we assumed $v\in[0,r]$, the size of $I^{\mathrm{L}}_\rho$ is greater than a global constant depending only on $r$, so that its dependence on $\rho$ is $O(1)$.

The proof is completed by inspecting the $\rho$-coefficients above.
\end{proof}


\end{document}